\documentclass[reqno]{amsart}
 \usepackage{amssymb}
 \usepackage[usenames, dvipsnames]{color}
 \usepackage{hyperref}
 \usepackage{verbatim}
 \usepackage{todonotes}
 \usepackage{cite}

 \allowdisplaybreaks
 
 \def\bR{\mathbb{R}}

 \newcommand{\dv}{\operatorname{div}}

 \allowdisplaybreaks

 \theoremstyle{plain}
 \newtheorem{theorem}{Theorem}[section]
 
 \newtheorem{lemma}[theorem]{Lemma}

 \theoremstyle{definition}
 \newtheorem{definition}[theorem]{Definition}

 \theoremstyle{remark}

 \numberwithin{equation}{section}

\begin{document}
 	
 	\title[five dimensional stationary Navier-Stokes equations]{Local $\varepsilon$-regularity criteria for the five dimensional stationary Navier-Stokes equations}
 	
 	\author[X. Cui]{Xiufang Cui}
 	\address[X. Cui]{School of Mathematical Sciences, Shanghai Jiao Tong University, Shanghai, 200240,  China}
 	\email{cuixiufang@sjtu.edu.cn}
 \thanks{ }

 	\begin{abstract}
 		We establish  interior and boundary $\varepsilon$-regularity criteria  at one scale for suitable weak solutions to the five dimensional stationary incompressible Navier-Stokes equations 
  which improve previous results in \cite{K} and \cite{Str}. Our proof is based on an iteration argument, Campanato's method, and interpolation techniques.
 	\end{abstract}
 	
 	\maketitle
 	
 	\section{Introduction}
 	In this paper,  we consider the five dimensional stationary incompressible Navier-Stokes equations with unit viscosity:
 	\begin{align}\label{1.1}
 		\begin{cases}
 			u\cdot\nabla u -\Delta
 			u+\nabla p=f\\
 			\nabla\cdot u=0
 		\end{cases}
 		\quad \text{in}\ \Omega,
 	\end{align}
 	where $u=\big(u_1(x),\dots , u_5(x)\big)\in \mathbb{R}^5$ is the velocity field, $p=p(x)\in \mathbb{R}$ is the pressure and $f=\big(f_1(x), \dots, f_5(x)\big) \in \mathbb{R}^5$ is the given external force.  We prove that any suitable weak solution $u$  of \eqref{1.1} is locally H\"older continuous  in $\Omega$ when  certain scale invariant quantities are small. Here $\Omega$ is either the unit ball $B_1$ or the unit upper half ball $B_1^+$ in $\mathbb{R}^5$. In the upper half ball case, we impose the zero Dirichlet boundary condition on a flat portion of the boundary
 	\begin{align}\label{1.2}
 		u=0\quad \text{on}\   \partial \Omega\cap\{x_5=0\}.
 	\end{align}
 	
The  regularity problem of  the Navier-Stokes equations  is an important question in fluid dynamics, and  much effort has been made in the literature.  In \cite{Hopf51, Leray34},  Leray and Hopf proved the existence of weak solutions to the  time-dependent three dimensional  incompressible Navier-Stokes equations for any given $L^2$ initial data. However, the regularity  and uniqueness of such weak solutions are still open. Many criteria which ensure  the regularity of weak solutions have been developed  with various conditions on  the velocity, the vorticity, or  the pressure, and  the most famous one is the Lady\v{z}enskaja-Prodi-Serrin criteria, see \cite{BKM, BG02, Leray34, Prodi59, Serrin63} and the references therein.  On the other hand, there is an extensive literature  devoting to estimating the size of   possible singularity set.  
In  a series of papers \cite{S1,S2,S3}, Scheffer studied the Hausdorff measure of possible singular set for weak solutions of  three dimensional Navier-Stokes equations.  In \cite{CKN},   Caffarelli, Kohn, and Nirenberg  introduced the notion of suitable weak solutions and proved that the one dimensional Hausdorff measure of  singular set is zero for any suitable weak solution of three dimensional Navier-Stokes equations. A simplified proof was given  later by Lin in \cite{Lin}  when the external force $f$ is zero and the pressure $p\in L^{3/2}(Q_1)$.  In \cite{TX}, Tian and Xin proved  a local uniform $L^\infty$ bound of $\nabla u$ for smooth solutions when either the scaled local  $L^2$ norm of the vorticity or the scaled local total energy is small. In   \cite{V}, Vasseur gave a new proof of the result in \cite{CKN} by using the De Giorgi iteration. In particular, he proved that if the quantity
 	$$
 	\sup_{t\in [-1,0]}\int_{B_1}|u(t,x)|^2\,dx + \int_{Q_1}|\nabla u|^2\, dxdt+\int_{-1}^{0}\|p\|^{\tilde q}_{L^1_{x}(B_1)}\, dt
 	$$
 	is sufficiently small, then $u$ is regular in $B_{1/2}$.
 	The above result was recently extended to the boundary case in \cite{HK}, where it was also proved that any suitable weak solution $u$ of the three dimensional incompressible  Navier-Stokes equations is regular in $\overline{Q_{1/2}^+}$ provided that
 	$$
 	\|u\|_{L^q(Q_1^+)}+\|p\|_{L_t^{\tilde q}L_x^1(Q_1^+)}
 	$$
 	is sufficiently small,  where $q>5/2$ and $\tilde q>1$.

 	Considerable  attention has also been paid to the stationary Navier-Stokes equations. In \cite{Ge79}, the author proved the existence of regular solutions to   the $n$ dimensional stationary incompressible Navier-Stokes equations, where the integer $n\in\{2,3,4\}$. See also \cite{Tsai18}. Since five is the smallest dimension in which the stationary Navier-Stokes equations are  super-critical,  there is a great number of  papers devoted to this case. See, for instance,  \cite{K, Str, Struwe2} and the references therein. In \cite{Str}, for the five dimensional  stationary incompressible Navier-Stokes equations \eqref{1.1},   Struwe  proved that any suitable weak solution   $u$ is   H\"older continuous near $x\in \Omega$     provided  that
 	\begin{equation}
 	\label{eq1.09}
 	\limsup_{r\rightarrow 0}\frac{1}{r}\int_{B_r(x)}|\nabla u|^2\ dx
 	\end{equation}
 	is small and the external force $f\in L^q(\Omega)$,  where $q>5/2$ and $\Omega\subset \mathbb{R}^5$ is an open set. The above result was extended to the boundary  case by Kang in \cite{K} under the condition that either
 	\begin{equation}\label{eq8.08}
 	\limsup_{r\rightarrow 0}\frac{1}{r}\int_{B(x, r)\cap\Omega}|\nabla u|^2\ dx=0
 	\end{equation}
 	or
 	\begin{equation} \label{eq8.09}
 	\liminf_{r\rightarrow 0}\frac{1}{r^2}\int_{B(x, r)\cap\Omega}|u|^3\ dx=0,
 	\end{equation}
 	where  $x\in \partial \Omega$ and  $\Omega $ is any smooth domain in $\bR^5$.   For more relevant research about  the  higher dimensional stationary Navier-Stokes equations,  we refer readers  to \cite{DG, DS,F1, F2, F3, F4,JS, Ko, LY21} and the references therein. 

 	In this paper, we use an iteration argument and interpolation technique to  establish local $\varepsilon$-regularity criteria at one scale for suitable weak solutions to the five dimensional stationary incompressible Navier-Stokes equations  in either the  ball  $B_1$ or the  upper half ball $B_1^+$.
 	Our main result, Theorems \ref{Them 1}, reads that a suitable weak solution $u$ in $B_1$ is regular in $B_{1/2}$ provided that
 	$$
 	\|u\|_{L^q(B_1)} +\|f\|_{L^{2}(B_1)}
 	$$
 	is sufficiently small,  where $q$ is any exponent great than $5/2$.
 	In Theorem \ref{Them 2},  we obtain a similar result near the boundary. Compared to the conditions \eqref{eq1.09}, \eqref{eq8.08}, and \eqref{eq8.09} in the previous papers \cite{K} and \cite{Str},   our smallness condition on the $L^q$ norm of $u$ is substantially weaker.
 	
 	Let us give an outline of the proof.  Following the ideas in \cite{HK},  we use Campanato's characterization of H\"older continuity and  an iteration argument  to prove the regularity of $u$. We decompose $u$ as $u=w+v$, where $v$ is  a harmonic function.  We  estimate $w$ by the $L^p$ estimates of elliptic equations  and the  uniform decay rates of certain scale invariant quantities.  For $v$, we  use the properties of harmonic functions. To be  more precise,  we  first show that  the  values of  the scale invariant quantities $A+E$, $A^++E^+$, $G$, and  $G^+$ in  a small ball can be  controlled by their values  in a larger ball. We refer the reader to Section 2 for the definitions of $A+E$, and other scale invariant quantities. This part of the argument is standard. See, for instance, \cite{GKT06, GKT07, Lin}. Then we  derive the smallness of $E$ and $E^+$ in Lemmas \ref{lem 3.3} and \ref{lem 4.3} under the conditions of Theorems \ref{Them 1} and \ref{Them 2},  respectively.  This is a key estimate in the proofs.  Based on the above results, we  show certain uniform decay rates of the scale invariant quantities by induction in Lemmas \ref{lem 3.4} and \ref{lem 4.4}. Combining the  uniform decay rates and the  $L^p$ estimates for elliptic equations, we obtain  the H\"older continuity of $u$ in  either $B_{1/2}$ or $B^+_{1/2}$
  by using Campanato's characterization of H\"older continuity.
 	Compared to the interior case, the main obstacle in showing the $\varepsilon$-regularity criteria in the upper half ball $B_1^+$ is to estimate  the pressure  term   up to the  boundary. More precisely,  when we estimate the quantity $G$ in $B_1$, the pressure is decomposed as the sum of a harmonic function and a term which can be controlled by the Calder\'on-Zygmund estimate. However, this method does not seem to work for the boundary case. In the upper half ball $B_1^+$, we adopt the pressure decomposition originally due to Seregin \cite{Se3}, and then apply the known solvability and boundary regularity results for the linear Stokes systems (see \cite{Ga94, K, MS, Se2}).  Compared to \cite{HK},  we only need the smallness assumption of $u$ and  $f$  thanks to the local $W^{k,p}$ estimate for linear Stokes systems. Besides, due to the presence of  the external force $f$ and the  different Sobolev embedding inequality in  five dimensions, it is  quite delicate to choose suitable  exponents for the scale invariant quantities.
 	
 	The main results of this paper are stated as follows.   The notation in Theorems \ref{Them 1} and  \ref{Them 2} is introduced in Section 2.

 	\begin{theorem}\label{Them 1} Let  $\Omega=B_1$ and the pair $(u,p)$ be a suitable weak solution to the five dimensional stationary incompressible Navier-Stokes equations \eqref{1.1}.  Let  $p\in L^1(\Omega)$ and $f\in L^q(\Omega)$ for some $q\in(5/2, 10/3)$. There exists a positive constant $\varepsilon$   such that if
 		\begin{align}\label{1.3b}
 			\int_{B_1}|u|^q\ dx + \int_{B_1}|f|^{2}\ dx <\varepsilon,
 		\end{align}
 		then $u$ is regular in $B_{1/2}$.
 	\end{theorem}
 
 	\begin{theorem}\label{Them 2}  Let  $\Omega=B_1^+$ and the pair $(u,p)$ be a suitable weak solution to the five dimensional stationary incompressible Navier-Stokes equations \eqref{1.1}  with the boundary condition \eqref{1.2}. Let  $p\in L^1(\Omega)$ and $f\in L^q(\Omega)$ for some $q\in (5/2, 10/3)$. There exists a positive constant $\varepsilon$ such that if
 		\begin{align}\label{1.4b}
 			\int_{B_1^+}|u|^q\ dx + \int_{B_1^+}|f|^{2}\ dx <\varepsilon,
 		\end{align}
 		then $u$ is regular in  $\overline{B^+_{1/2}}$.
 	\end{theorem}
 	
In this paper, we only consider the flat boundary for simplicity. The boundary regularity result in Theorem \ref{Them 2} still holds true for general $C^2$ boundary by following the argument, for example, in \cite{SSS}. We emphasize that  the strict inequality $q>5/2$  is necessary when we derive the smallness of $E$ and $E^+$  in Lemmas \ref{lem 3.3} and \ref{lem 4.3} and the uniform decay rates for the scale invariant quantities in Lemmas \ref{lem 3.4} and \ref{lem 4.4}. At the time of this writing, it is not clear to us whether  the result still holds when $q=5/2$. Besides,  even though our results refine the  known regularity criteria, they probably do not  improve the estimate of the  Hausdorff dimension of the singular set.  

 	The remainder of this paper is organized as follows. In Section 2, we introduce some notation and state  a lemma which will be frequently used in the proof of the main results. The interior  $\varepsilon$-regularity for suitable weak solutions is proved in Section 3. We prove the corresponding boundary  $\varepsilon$-regularity in Section 4.  Throughout this paper, we use  $N$ to denote various constants which may change from line to line. We also use the expression $N=N(\cdot\cdot\cdot)$ for the constant which depends on the contents  between the parentheses.

 	\section{Preliminaries}
 	In this section, we introduce some notation  and the definition of suitable weak solutions which will be used throughout the paper. We also present a key lemma which will be  used  to prove our main results.
 	
 	For any point $x_0=(x_0^1, x_0^2, \dots, x_0^5)\in \mathbb{R}^5$, we   use  the   notation
 	\begin{align*}
 		B(x_0,\rho)=\big\{x\in \mathbb{R}^5: \ |x-x_0|<\rho\big\}\quad \text{and}\quad
 		B^+(x_0,\rho)=B(x_0,\rho)\cap \mathbb{R}^5_+
 	\end{align*}
 	to denote the  balls in $\mathbb{R}^5$ and half balls in $\mathbb{R}^5_+$ with  center at  $x_0$ and radius $\rho>0$.   For the convenience of notation, we denote
 	\begin{align*}
 		B_\rho=B(0,\rho), \quad B_\rho^+=B^+(0,\rho). 
 	\end{align*}
 	For any nonempty open set $\Omega\in \bR^5$, we use  the abbreviation
 	\begin{align*}
 		(u)_{\Omega}=\frac{1}{|\Omega|}\int_{\Omega}u \ dx
 	\end{align*}
 	to denote the average of $u$ in $\Omega$,  where $|\Omega|$ as usual denotes the Lebesgue measure of $\Omega$.
 	
 	We introduce the following    scale invariant quantities:
 	\begin{align*}
 		A(x_0, r)&=\frac{1}{r^{3}}\int_{B(x_0, r)}|u|^2\ dx,\\
 		C(x_0, r)&=\frac{1}{r^{5-q}}\int_{B(x_0, r)}|u|^q\ dx,\\
 		E(x_0, r)&=\frac{1}{r}\int_{B(x_0,r)}|\nabla u|^2\ dx,\\
 		G(x_0, r)&=\frac{1}{r^\frac{10-15c}{4-c}}\int_{B(x_0, r)}|p-(p)_{B(x_0, r)} |^{\frac{5(1+c)}{4-c}}\ dx,\\
 		P(x_0, r)&= \frac{1}{r^3}\int_{B(x_0, r)}|p-(p)_{B(x_0, r)}| \ dx,\\
 		F(x_0, r)&= r\int_{B(x_0, r)}|f|^{2}\ dx,
 	\end{align*}
 	where $  c\in \Big(\frac{5-q}{6q-5}, \frac14\Big)$ is any fixed positive constant and $ r\in (0,\rho]$.
 	Similarly, we also define  the following   scale invariant quantities:
 	\begin{align*}
 		A^+(x_0, r)&=\frac{1}{r^{3}}\int_ {B^+(x_0, r)}|u|^2\ dx,\\
 		C^+(x_0, r )&=\frac{1}{r^{5-q}}\int_{B^+(x_0, r)}|u|^q\ dx,\\
 		E^+(x_0, r )&=\frac{1}{r}\int_{ B^+(x_0,r)}|\nabla u|^2\ dx,\\
 		G^+(x_0, r )&=\frac{1}{r^\frac{10-15c}{4-c}}\int_{B^+(x_0, r)}|p-(p)_{B^+(x_0, r)} |^{\frac{5(1+c)}{4-c}}\ dx,\\
 		P^+(x_0, r )&= \frac{1}{r^3}\int_{B^+(x_0, r)}|p-(p)_{B^+(x_0, r)}| \ dx,\\
 		F^+(x_0, r)&= r\int_{B^+(x_0, r)}|f|^{2}\ dx,
 	\end{align*}
 	where $ c\in \Big(\frac{5-q}{6q-5},\frac14\Big)$ is any fixed positive  constant and $r\in (0,\rho]$.

 	For simplicity, we often denote $A(r)=A(0, r)$ and similarly for  other scale invariant quantities.
 	We see that all these quantities are invariant under the natural scaling:
 	\begin{align}
 		\label{eq10.23}
 		u_\lambda:=\lambda u(\lambda x), \quad p_\lambda:=\lambda^2p(\lambda x), \quad f_\lambda :=\lambda^3f(\lambda x),
 	\end{align}
 	where $\lambda>0$ is any constant.
 	
 	Let $\Omega$ be a domain in $\mathbb{R}^5$ and  $\Gamma\in \partial\Omega$. We denote $\dot{C}_0^\infty(\Omega, \Gamma)$ to be the space of divergence-free infinitely differentiable vector fields which vanishes near $\Gamma$. Let $J(\Omega, \Gamma)$ be the closure of $\dot{C}_0^\infty(\Omega, \Gamma)$ in $H^1(\Omega)$.

 	In this paper, we focus on an important type of  solutions called suitable weak solutions. The definition of suitable weak solutions was first introduced in a celebrated paper \cite{CKN} by Caffarelli, Kohn, and Nirenberg.
 	\begin{definition} We say that a pair   $(u, p)$  is a   suitable weak solution to the stationary Navier-Stokes equations  \eqref{1.1} in $\Omega$ vanishing on $\Gamma$ if $(u, p)\in J(\Omega, \Gamma)\times L^1(\Omega)$  satisfies \eqref{1.1} in the  sense of distribution and  for any non-negative function $\psi\in C^\infty(\overline{\Omega})$ vanishing in a neighborhood of the boundary $\partial\Omega\backslash\Gamma$, we have the local energy inequality
 		\begin{align}\label{2.1}
 			2\int_{\Omega}|\nabla u|^2\psi \ dx\leq \int_{\Omega}\Big[|u|^2\Delta \psi+\big(|u|^2+2p\big)u\cdot\nabla\psi +2f\cdot u\psi\Big]\ dx.
 		\end{align}
 	\end{definition}
 	The following  key lemma  will be frequently  used in the proofs of our main results.
 	\begin{lemma}\label{lem 4.5}
 		Suppose $f(\rho_0)<N_0$. If there exist $\alpha>\beta>0$ and $N_1, N_2>0$ such that for any $0<r<\rho\leq \rho_0$, it holds that
 		\begin{align*}
 			f(r)\leq N_1  (r/\rho)^\alpha f(\rho)+N_2\rho^\beta,
 		\end{align*}
 		then there exist positive constants $N_3$ and $N_4$ depending on $N_0, N_1, N_2, \alpha$ and $ \beta$ such that
 		\begin{align*}
 			f(r)\leq N_3 (r/\rho_0)^\beta f(\rho_0)+N_4 r^\beta,
 		\end{align*}
 		where  $r\in(0,\rho_0]$ is any constant.
 	\end{lemma}
 	\begin{proof}
 		See, for example, Lemma 2.1 in Chapter 3 of \cite{MG}.
 	\end{proof}
 	
 	\section{Interior \texorpdfstring{$\varepsilon$}{}-regularity}
 	This section is devoted to the proof of Theorem \ref{Them 1}. According to  Campanato's characterization of H\"older continuity, we  only need to prove
 	\begin{align*}
 		\int_{B(x, \rho)}|u-(u)_{B(x, \rho)}|^2 \ dx\leq N \rho^{5+\alpha}
 	\end{align*}
 	for some index $\alpha\in (0,1)$.   To this end,  we decompose $u$ as $u=w+v$, where $v$ is a harmonic function. We estimate the  mean oscillation of $w$ by the $L^p$ estimates for  elliptic equations  and the uniform  decay estimates of the scale invariant quantities, which in turn is proved by an induction argument.  To obtain the above decay estimates, we need all these scale invariant quantities to be small. However,  from the assumptions in  Theorem \ref{Them 1}, we only have the smallness of $C$ and $F$.   For this, we apply the local $W^{k,p}$ estimate for linear Stokes systems  to get the smallness of $P$,  and then show the smallness of $E$ by an iteration argument based on the  fact that the  values of  $A+ E$ and $G$ in  a smaller ball can be controlled by  their values in a larger ball.
 	For the mean oscillation of $v$,   we   use the  properties of harmonic  functions. More precisely, we first obtain the smallness of $E(3/4)$ in Lemma \ref{lem 3.3} by using an iteration argument as well as the estimates in  Lemmas \ref{lem 3.1} and   \ref{lem 2}. Then, based on the above results and the condition \eqref{1.3b},  we derive the uniform decay estimates of $A+E$, $G$, and $P$ in Lemma \ref{lem 3.4}.  Finally, by the   decay estimates  of these scale invariant quantities and the $L^p$ estimate for elliptic equations,  we prove the H\"older continuity of $u$ in the ball $B_{1/2}$ by using Campanato's characterization of H\"older continuity.
 	
 	Throughout this section, we use  the pair $(u, p)$ to represent  a suitable weak solution to the incompressible Navier-Stokes equations \eqref{1.1}.
 	
 	First of all, we rewrite the equation into
 	$$
 	-\Delta u+\nabla p=f-\dv(u\otimes u).
 	$$
 	Applying the local $W^{k,p}$ estimate for linear Stokes systems (see, for instance, in \cite[Theorem 3.8]{K}), we have
 	$$
 	\|p-(p)_{B_{1/2}}\|_{L^{q/2}(B_{1/2})}\le N\|u\|^2_{L^{q}(B_1)}+N\|f\|_{L^{q/2}(B_1)}
 	+N\|u\|_{L^{1}(B_1)}.
 	$$
 	Therefore, by using \eqref{1.3b} and after a scaling, we may assume that
 	\begin{align}\label{1.3}
 		\int_{B_1}|u|^q\ dx+\int_{B_1}|p-(p)_{B_{1}}|\ dx + \int_{B_1}|f|^{2}\ dx <\varepsilon,
 	\end{align}
 	
 	In the next lemma, we show  that the values of $A+ E$ and $G$ in smaller ball can be controlled  by their values in a larger ball by the following two lemmas.
 	
 	\begin{lemma}\label{lem 3.1}  For any  $\gamma\in (0,1/2]$ and $B(x_0, \rho)\subset B_1$,  we have
 		\begin{align}\label{3.1}
 			&A(x_0, \gamma\rho)+E(x_0, \gamma\rho)\notag\\
 			&\leq N\Big[\gamma^2 A(x_0, \rho)+\gamma^{-2}A(x_0, \rho)^\frac14E(x_0, \rho)^\frac{5}{4}+\gamma^{-2}A(x_0, \rho)^\frac32\nonumber\\
 			&\quad + \gamma^{-2}    A(x_0, \rho)^\frac{9c-1}{4(1+c)}E(x_0, \rho)^\frac{3-7c}{4(1+c)}G(x_0, \rho)^\frac{4-c}{5(1+c)}\nonumber\\
 			&\quad +\gamma^{-2}A(x_0, \rho)^\frac12G(x_0, \rho)^\frac{4-c}{5(1+c)}+\gamma^{-4} F(x_0, \rho)\Big],
 		\end{align}
 		where $N=N(c)$ is a positive constant independent of $\gamma$ and $\rho$.
 	\end{lemma}
 	\begin{proof}
 		The proof is more or less standard. For the sake of completeness, we give the detail of proof. By the  scale invariant property, we may assume $\rho=1$.
 		
 		For the test function
 		\begin{align*}
 			\Gamma(x)=  \big(\gamma^2+|x-x_0 |^2\big)^{-\frac32},
 		\end{align*}
 		one has
 		\begin{align*}
 			\Delta \Gamma (x) <0 \quad \mathrm{in} \ \mathbb{R}^5, \quad \Delta\Gamma (x) \leq -N\gamma ^{-5}\quad \mathrm{in} \ B(x_0, \gamma).
 		\end{align*}

 		We choose a suitable smooth cutoff function  $\phi (x)\in C_0^\infty(B(x_0, 1))$ which  satisfies
 		\begin{align*}
 			0\leq \phi(x)\leq 1 \quad \mathrm{in}\ B(x_0, 1), \quad \phi (x)=1\quad  \mathrm{in}\ B(x_0, 1/2),\\
 			|\nabla \phi (x)|\leq N, \quad |\nabla^2 \phi (x)|\leq N \quad\mathrm {in} \ B(x_0, 1).\qquad
 		\end{align*}
 		A simple calculation yields
 		\begin{align}\label{3.2}
 			\Gamma (x)\phi (x)\geq N\gamma^{-3},  \quad \forall \ x\in B( x_0, \gamma),
 		\end{align}
 		\begin{align}\label{3.3}
 			|\Gamma (x)\phi (x)|\leq N\gamma^{-3},\quad \forall \ x\in B( x_0, 1),
 		\end{align}
 		\begin{align}\label{3.4}
 			|\nabla\Gamma(x)\phi (x)+\Gamma(x)\nabla \phi(x)|\leq N \gamma^{-4},\quad \forall\  x\in B(x_0, 1)
 		\end{align}
 		and
 		\begin{align}\label{3.5}
 			|\Gamma(x)\Delta\phi(x)+2\nabla\Gamma(x)\cdot\nabla \phi(x)|\leq N, \quad \forall\  x\in B(x_0, 1).
 		\end{align}

 		Taking $\psi=\Gamma(x) \phi(x) $ in the energy inequality  \eqref{2.1},   we have
 		\begin{align*}
 			- \int_{B(x_0, \gamma)} &|u|^2 \Delta\Gamma(x)  \phi(x)\ dx +2\int_{B(x_0, \gamma)}|\nabla u|^2\Gamma(x)\phi (x) \ dx\\
 			\leq & \int_{B(x_0, 1)} |u|^2|\Gamma(x)\Delta\phi (x) +2\nabla\Gamma(x)\cdot\nabla \phi(x) |\ dx\\
 			&+\int_{B(x_0, 1)}\Big(|u|^2+|p-(p)_{B(x_0, 1)}|\Big)|u| \cdot|\nabla\Gamma(x)\phi (x)+\Gamma(x)\nabla\phi(x)|\ dx\\
 			& +2\int_{B(x_0, 1)}|f\cdot u ||\Gamma(x) \phi (x)|\  dx.
 		\end{align*}
 		By the properties \eqref{3.2}-\eqref{3.5}   and Young's inequality,  one obtains
 		\begin{align*}
 			\gamma^{-5}\int_{B(x_0,\gamma)}& |  u|^2 \  dx+\gamma^{-3}\int_{B(x_0, \gamma)} |\nabla u|^2  \ dx\\
 			\leq &
 			N\Big[\int_{B(x_0, 1)}|u|^2 \ dx+\gamma^{-4}\int_{B(x_0, 1)}\Big(|u|^2+|p-(p)_{B(x_0, 1)}|\Big)|u|\ dx\\ &+\gamma^{-6}\int_{B(x_0, 1)}f^2\ dx \Big],
 		\end{align*}
 		which  yields
 		\begin{align}\label{3.6}
 			& A(x_0,\gamma)+ E(x_0, \gamma)\leq N\Big[\gamma^2A(x_0, 1)\nonumber\\
 			&\quad +\gamma^{-2}\int_{B(x_0, 1)}\Big(|u|^2+|p-(p)_{ B(x_0,1)}|\Big)|u|\ dx+\gamma^{-4} F(x_0, 1)\Big].
 		\end{align}

 		For the second term on the right-hand side of \eqref{3.6},  by H\"older's inequality  and the Sobolev embedding inequality, we have
 		\begin{align*}
 			\int_{B(x_0, 1)}|u|^3\ dx&\leq \Big(\int_{B(x_0, 1)}|u|^\frac{10}{3} \ dx\Big)^\frac34\Big(\int_{B(x_0, 1)}|u|^2\ dx\Big)^\frac14\\
 			&\leq N \Big(\int_{B(x_0, 1)} |\nabla u|^2\ dx+\int_{B(x_0, 1)} |u|^2 \  dx\Big)^\frac54\Big(\int_{B(x_0, 1)}|u|^2\ dx\Big)^\frac{1}{4}\\
 			&\leq N \Big(A(x_0, 1)^\frac{1}{4}E(x_0, 1)^\frac{5}{4}+A(x_0, 1)^\frac{3}{2}\Big)
 		\end{align*}
 		and
 		\begin{align*}
 			&\int_{B(x_0, 1)}|p-(p)_{B(x_0, 1)}|| u |\ dx\\
 			&\leq \Big(\int_{B(x_0, 1)}|p-(p)_{B(x_0,1)}|^\frac{5(1+c)}{4-c} \ dx\Big)^\frac{4-c}{5(1+c)}\Big(
 			\int_{B(x_0, 1)}|u|^\frac{5(1+c)}{1+6c}\ dx\Big )^\frac{1+6c}{5(1+c)}\\
 			&\leq N\Big(\int_{B(x_0, 1)}|p-(p)_{B(x_0, 1)}|^\frac{5(1+c)}{4-c} \ dx\Big)^\frac{4-c}{5(1+c)}\Big(
 			\int_{B(x_0, 1)}|u|^2\ dx\Big )^\frac{9c-1}{4(1+c)}\\
 			&\quad \cdot\Big(
 			\int_{B(x_0, 1)}|\nabla u|^2\ dx+\int_{B(x_0, 1)}|u|^2\ dx\Big )^\frac{3-7c}{4(1+c)}\\
 			&\leq N\Big(A(x_0, 1)^\frac{9c-1}{4(1+c)}E(x_0, 1)^\frac{3-7c}{4(1+c)}G(x_0, 1)^\frac{4-c}{5(1+c)}+A(x_0, 1)^\frac12G(x_0, 1)^\frac{4-c}{5(1+c)}\Big),
 		\end{align*}
 		where $N=N(c)$ is some positive constant.
 		Plugging above two inequalities into  \eqref{3.6},  the lemma is proved.
 	\end{proof}
 	
 	\begin{lemma}\label{lem 2} For any    $\gamma\in (0,3/5]$ and $B(x_0, \rho)\subset B_1$,  we have
 		\begin{align}\label{3.7}
 			G(x_0, \gamma\rho)\leq& N\gamma^{-\frac{10-15c}{4-c}} \Big(\gamma^{\frac{25}{4-c}}P(x_0, \rho)^\frac{5(1+c)}{4-c}+ A(x_0, \rho)^{\frac{5(1-4c)}{2(4-c)}}E(x_0, \rho)^\frac{5(1+6c)}{2(4-c)} \nonumber\\
 			&+F(x_0, \rho)^{\frac{5(1+c)}{2(4-c)}}\Big),
 		\end{align}
 		where $N=N(c)$ is a positive constant independent of $\gamma$ and $\rho$.
 	\end{lemma}
 	\begin{proof}
 		By the   scale invariant property, we assume $\rho=1$.
 		Taking divergence on  both sides of  the first equation  in
 		\eqref{1.1},  since  $u$ is divergence free,  we have
 		\begin{align}\label{3.8}
 			\Delta p&= -D_{ij}(u_iu_j)+\nabla \cdot f\nonumber\\
 			&=-D_{ij}\big[\big(u_i-(u_i)_{B(x_0, 1)}\big)\big(u_j-(u_j)_{B(x_0,1)}\big)\big]+\nabla \cdot f.
 		\end{align}
 		
 		Let $\eta(x)$ be a suitable smooth cutoff function in $B(x_0, 1)$ satisfying $0\leq \eta\leq 1$ and $\eta=1$ in $\overline{B(x_0, 2/3)}$.
 		We  decompose the pressure term as follows
 		\begin{align*}
 			p=\tilde p+\tilde h,
 		\end{align*}
 		where $ \tilde p $ is the Newtonian potential of
 		\begin{align*}
 			-D_{ij}\big[(u^i-(u^i)_{B(x_0, 1)})(u^j-(u^j)_{B(x_0, 1)})\eta(x-x_0)\big]+\nabla \cdot\big [f\eta(x-x_0)\big].
 		\end{align*}
 		By applying the  Calder\'on-Zygmund estimate, the   Sobolev embedding inequality, and  H\"older's inequality,  one has
 		\begin{align}\label{3.9}
 			&\int_{B(x_0, 1)}| \tilde p |^\frac{5(1+c)}{4-c}\ dx\nonumber\\
 			&\leq N\Big[ \int_{B(x_0, 1)}|u-(u)_{B(x_0,1)}|^\frac{10(1+c)}{4-c}\ dx+ \int_{B(x_0, 1)}\big|\Delta^{-1}\nabla\cdot [f\eta(x-x_0 )]\big|^{\frac{5(1+c)}{4-c}}\ dx\Big] \nonumber\\
 			&\leq N\Big[\Big(\int_{B(x_0, 1)}|u|^2\ dx\Big)^\frac{5(1-4c)}{2(4-c)} \Big(\int_{B(x_0, 1)}|\nabla u|^2\ dx\Big)^\frac{5(1+6c)}{2(4-c)}\notag\\
 			&\qquad +\Big(\int_{B(x_0, 1)}|f|^{1+c} \ dx\Big)^{\frac{5}{4-c}}\Big]\nonumber\\
 			&\leq N\Big[\Big(\int_{B(x_0, 1)}|u|^2\ dx\Big)^\frac{5(1-4c)}{2(4-c)}\Big(\int_{B(x_0, 1)}|\nabla u|^2\ dx\Big)^\frac{5(1+6c)}{2(4-c)}\nonumber\\
 			&\qquad +\Big(\int_{B(x_0, 1)}|f|^{2}\ dx\Big)^{\frac{5(1+c)}{2(4-c)}}\ \Big],
 		\end{align}
 		where $N=N(c)$.
 		
 		Since $ \tilde h $ is a harmonic function in $B(x_0,  2/3)$,   any Sobolev norm of $ \tilde h $ in  $B(x_0, \gamma)$ can be controlled by the $L^p$ norm of it in $B(x_0, 2/3)$, where $p\in [1,+\infty]$. Thus, we have
 		\begin{align}\label{3.10}
 			&\int_{B(x_0,\gamma)} | \tilde h-( \tilde h )_{B(x_0, \gamma)}|^\frac{5(1+c)}{4-c}\ dx\nonumber\\
 			&\leq N\gamma^{\frac{5(1+c)}{4-c}}\int_{B(x_0, \gamma)}|\nabla  \tilde h |^\frac{5(1+c)}{4-c} \ dx
 			\leq N\gamma^{\frac{25}{4-c}}\sup_{B(x_0,\gamma)}|\nabla  \tilde h  |^\frac{5(1+c)}{4-c}\nonumber\\
 			&\leq N\gamma^{\frac{25}{4-c}} \Big(\int_{B(x_0,  2/3)}| \tilde h -(p)_{B(x_0, 1)}|\ dx\Big)^\frac{5(1+c)}{4-c}\nonumber\\
 			&\leq N\gamma^{\frac{25}{4-c}} \Big[ \Big(\int_{B(x_0, 1)}|p-(p)_{B(x_0, 1)}|\ dx\Big)^\frac{5(1+c)}{4-c}+ \int_{B(x_0,1)}| \tilde p |^\frac{5(1+c)}{4-c} \ dx\Big].
 		\end{align}
 		The combination of \eqref{3.9} and \eqref{3.10} yields
 		\begin{align*}
 			&\int_{B(x_0,\gamma)}|p-(p)_{B(x_0, \gamma)}|^\frac{5(1+c)}{4-c}  \ dx\\
 			&\leq  N\Big(\int_{B(x_0,\gamma)}| \tilde h -( \tilde h )_{B(x_0, \gamma)}|^\frac{5(1+c)}{4-c} \ dx+\int_{B(x_0,\gamma)}| \tilde p |^\frac{5(1+c)}{4-c} \ dx\Big)\\
 			&\leq N\Big[\gamma^{\frac{25}{4-c}} \Big(\int_{B(x_0, 1)}|p-(p)_{B(x_0, 1)}|\ dx\Big)^\frac{5(1+c)}{4-c}+(1+\gamma^\frac{25}{4-c}) \int_{B(x_0, 1)}| \tilde p |^{\frac{5(1+c)}{4-c}} \ dx \Big]\\
 			&\leq N\Big[\gamma^{\frac{25}{4-c}} \Big(\int_{B(x_0, 1)}|p-(p)_{B(x_0,1)}|\ dx\Big)^\frac{5(1+c)}{4-c}+\Big(\int_{B(x_0, 1)}|u|^{2} \ dx \Big)^{\frac{5(1-4c)}{2(4-c)}}\\
 			&\quad \cdot \Big(\int_{B(x_0, 1)}|\nabla u|^2 \ dx \Big)^\frac{5(1+6c)}{2(4-c)} +\Big(\int_{B(x_0, 1)}|f|^{2}\ dx\Big)^{\frac{5(1+c)}{2(4-c)}}  \Big],
 		\end{align*}
 		where  $N=N(c)$. The conclusion of Lemma \ref{lem 2} follows immediately.
 	\end{proof}

 	The following lemma  shows that   $E(3/4)$ can be  sufficiently  small  under the conditions of Theorem \ref{Them 1}.  This lemma is  needed for us to derive the decay estimates of   scale invariant quantities with respect to the radius.
 	
 	\begin{lemma}\label{lem 3.3}
 		Let $\Omega=B_1$  and $(u,p)$ be a suitable weak solution to \eqref{1.1}  satisfying the condition of Theorem \ref{Them 1} Then  we have
 		\begin{align*}
 			E\big(3/4\big)\leq   N\varepsilon^\beta,
 		\end{align*}
 		where $ N$ and $\beta$ are positive  constants which depend only on $c$ and $q$.
 	\end{lemma}
 	\begin{proof}
 		Let  $\rho_k=1-2^{-(k+1)}$ and $B_k=B_{\rho_k}$, where   $k$ is any positive integer. We choose  cutoff functions $\psi_k$  which satisfy
 		\begin{align*}
 			\mathrm{supp} \ \psi_k\subset B_{k+1},\quad \psi_k=1\quad \mathrm{in}\ B_k,\qquad\nonumber\\
 			|D\psi_k|\leq N 2^k,\quad | D^2\psi_k| \leq N 2^{2k}\quad \mathrm{in}\ B_{k+1}.
 		\end{align*}
 		By the energy inequality \eqref{2.1}, we have  \begin{align}\label{3.11}
 			&2\int_{B_k}|\nabla u|^2  \ dx\notag\\
 			&\leq N \int_{B_{k+1}}\Big(2^k|u|^3+2^{k+1}|p-(p)_{B_{k+1}}||u|  +2^{2k}|u|^2+2|f||u | \Big)\ dx.
 		\end{align}
 		Due to  the  condition \eqref{1.3} and H\"older's inequality, we have
 		\begin{align}\label{AC}
 			A(\rho_{k+1})\leq NC(\rho_{k+1})^\frac{2}{q}\leq N\varepsilon^\frac{2}{q}.
 		\end{align}

 		Next we use \eqref{1.3} and \eqref{AC} to estimate each term on the right-hand side of \eqref{3.11}.

 		For the first term on the right-hand side of \eqref{3.11},  by  H\"older's inequality,  the  Sobolev embedding inequality and \eqref{AC},  one has
 		\begin{align}\label{3.12}
 			&\int_{B_{k+1}}|u|^3 \ dx=\int_{B_{k+1}}|u|^a|u|^{3-a}\ dx\nonumber\\
 			&\leq \Big(\int_{B_{k+1}}|u|^q \ dx\Big)^\frac{a}{q}\Big(\int_{B_{k+1}}|u|^\frac{(3-a)q}{q-a} \ dx\Big)^\frac{q-a}{q}\nonumber\\
 			&\leq N \Big(\int_{B_{k+1}}|u|^q \ dx\Big)^\frac{a}{q}\Big(\int_{B_{k+1}}|u|^2 \ dx\Big)^\frac{q+3aq-10a}{4q} \Big(\int_{B_{k+1}}|\nabla u|^2 \ dx\nonumber\\
 			&\quad + \rho_{k+1}^{-2}\int_{B_{k+1}}|u|^2 \ dx\Big)^\frac{5(q-aq+2a)}{4q}\nonumber\\
 			&  \leq N\rho_{k+1}^2\Big( A(\rho_{k+1})^{\frac{q+3aq-10a}{4q}}C(\rho_{k+1})^{\frac{a}{q}}E(\rho_{k+1}) ^{\frac{5(q-aq+2a)}{4q}}\nonumber\\
 			&\quad + A(\rho_{k+1})^\frac{3-a}{2}C(\rho_{k+1})^\frac{a}{q}\Big)\notag\\
 			&\leq N\Big(\varepsilon^\frac{q+5aq-10a}{2q^2}E(\rho_{k+1})^\frac{5(q-aq+2a)}{4q}+\varepsilon^\frac{3}{q}\Big),
 		\end{align}
 		where  $N=N(a, q)$ is some positive  constant. Here, we take the constant $a\in (0, 1)$ to make sure that all the  exponents  in the above inequality are positive.

 		Similarly,  for the second term on the right-hand side of \eqref{3.11}, it   holds that
 		\begin{align}\label{3.13}
 			&\int_{B_{k+1}} |p-(p)_{B_{k+1}}|  |u| \ dx\nonumber\\
 			& \leq \Big(\int_{B_{k+1}}|p-(p)_{B_{k+1}}|^\frac{5(1+c)}{4-c} \ dx\Big)^\frac{4-c}{5(1+c)}\Big(
 			\int_{B_{k+1}}|u|^\frac{5(1+c)}{1+6c}\ dx\Big )^\frac{1+6c}{5(1+c)}\nonumber\\
 			&\leq \Big(\int_{B_{k+1}} |\nabla p  |^{1+c} \ dx\Big)^\frac{1}{1+c}\Big(\int_{B_{k+1}}|u|^2 \ dx\Big)^\frac{q+6qc-5-5c}{5(1+c)(q-2)}  \Big(\int_{B_{k+1}}|u|^q \ dx\Big)^\frac{3-7c}{5(1+c)(q-2)},
 		\end{align}
 		where the exponents $\frac{q+6qc-5-5c}{5(1+c)(q-2)}$ and  $\frac{3-7c}{5(1
 			+c)(q-2)}$ are positive due to the fact $  c\in \Big( \frac{5-q}{6q-5}, \frac14\Big)$.

 		For the first  term on the right-hand side of    \eqref{3.13}, by the $W^{1, p}$ estimate for elliptic equation \eqref{3.8}, we have
 		\begin{align}\label{3.14}
 			\Big(\int_{B_{k+1}} |\nabla p  |^{1+c} \ dx\Big)^\frac{1}{1+c}
 			\leq& N\Big[\Big(\int_{B_{k+2}}|f|^{1+c}\ dx\Big)^{\frac{1}{1+c}}+\Big(\int_{B_{k+2}}|u\cdot\nabla u|^{1+c}\ dx\Big)^{\frac{1}{1+c}}\nonumber\\
 			&+ (\rho_{k+2}-\rho_{k+1})^{-\frac{1+6c}{1+c}} \int_{B_{k+2}}|p-(p)_{B_{k+2}}| \  dx  \Big],
 		\end{align}
 		where $N=N(c)>0$ is some constant.
 		
 		For the second term on the right-hand side of the above inequality, by  H\"older's inequality and the Sobolev embedding inequality,  one obtains
 		\begin{align}\label{3.15}
 			&\| u\cdot\nabla u\|_{L^{1+c}(B_{k+2})}\nonumber\\
 			&\leq \|u\|_{L^\frac{2(1+c)}{1-c}(B_{k+2})}\|\nabla u\|_{L^2(B_{k+2})}\nonumber\\
 			&\leq  N\Big(\|\nabla u\|_{L^2(B_{k+2})}^\frac{5c}{1+c}+\rho_{k+2}^{-\frac{5c}{1+c}}\|  u\|_{L^2(B_{k+2})}^\frac{5c}{1+c}\Big)\|  u\|_{L^2(B_{k+2})}^\frac{1-4c}{1+c}\|\nabla u\|_{L^2(B_{k+2})}\nonumber\\
 			&\leq  N \Big(\|u\|_{L^2(B_{k+2})}^{\frac{1-4c}{1+c}}\|\nabla u\|_{L^2(B_{k+2})}^{\frac{1
 					+6c}{1+c}}+\rho_{k+2}^{-\frac{5c}{1+c}}\|u\|_{L^2(B_{k+2})}\|\nabla u\|_{L^2(B_{k+2})}\Big),
 		\end{align}
 		where $N=N(c)$ is some  positive constant.

 		The combination of \eqref{3.14} and \eqref{3.15} yields
 		\begin{align*}
 			&\Big( \int_{B_{k+1}} |\nabla p  |^{1+c} \ dx\Big)^\frac{1}{1+c}\\
 			&\leq  N\Big[\rho_{k+2}
 			^\frac{5(1-c)}{2(1+c)}\Big(\int_{B_{k+2}}|f|^2\ dx\Big)^\frac12+\Big(\int_{B_{k+2}}|u|^2\ dx\Big)^\frac{1-4c}{2(1+c)}\Big(\int_{B_{k+2}}|\nabla u|^2\ dx\Big)^\frac{1+6c}{2(1+c)}\\
 			&\quad +\rho_{k+2}^{-\frac{5c}{1+c}}\Big(\int_{B_{k+2}}|u|^2 \ dx\Big)^\frac12\Big(\int_{B_{k+2}}|\nabla u|^2 \ dx\Big)^\frac12\\
 			&\quad+2^{\frac{ (k+3)(1+6c)}{1+c}} \int_{B_{k+2}}|p-(p)_{B_{k+2}}| \ dx \Big]\\
 			&\leq N\Big[\rho_{k+2}^\frac
 			{2-3c}{1+c}F(\rho_{k+2})^\frac12+\rho_{k+2}^{\frac{2-3c}{1+c}}
 			A(\rho_{k+2})^\frac{1-4c}{2(1+c)}E(\rho_{k+2})^\frac{1+6c}{2(1+c)}
 			\\
 			&\quad +\rho_{k+2}^{\frac{2-3c}{1+c}}A(\rho_{k+2})^\frac 12E(\rho_{k+2})^\frac12
 			+ 2^{\frac{ (k+3)(1+6c)}{1+c}}\rho_{k+2}^3 P(\rho_{k+2}) \Big],
 		\end{align*}
 		where $N=N(c)>0$. Inserting  the above inequality into \eqref{3.13}, by \eqref{1.3} and \eqref{AC},   we have
 		\begin{align}\label{3.16}
 			&\int_{B_{k+1}} |p-(p)_{B_{k+1}}||u| \ dx\nonumber\\
 			& \leq N\Big(\rho_{k+2}^2A(\rho_{k+2})^\frac{q+6qc-5-5c}{5(1+c)(q-2)}C(\rho_{k+2})^\frac{3-7c}{5(1+c)(q-2)}F(\rho_{k+2})^\frac12 \nonumber\\
 			&\quad+\rho_{k+2}^2A(\rho_{k+2})^\frac{7q-8qc-20+30c}{10(1+c)(q-2)}C(\rho_{k+2})^\frac{3-7c}{5(1+c)(q-2)}E(\rho_{k+2})^\frac{1+6c}{2(1+c)}\nonumber\\
 			&\quad+\rho_{k+2}^2A(\rho_{k+2})^\frac{7q+17qc-20-20c}{10(1+c)(q-2)}C(\rho_{k+2})^\frac{3-7c}{5(1+c)(q-2)}E(\rho_{k+2})^\frac12\nonumber\\
 			&\quad+2^{\frac{ (k+3)(1+6c)}{1+c}}\rho_{k+2}^{\frac{3+8c}{1+5c}} A(\rho_{k+2})^\frac{q+6qc-5-5c}{5(1+c)(q-2)}C(\rho_{k+2})^\frac{3-7c}{5(1+c)(q-2)} P(\rho_{k+2})\Big)\notag\\
 			&\leq N\Big(\varepsilon^{\frac{q+2}{2q}}+\varepsilon^\frac{2-3c}{q(1+c)}E(\rho_{k+2})^\frac{1+6c}{2(1+c)}+\varepsilon^\frac{2}{q}E(\rho_{k+2})^\frac12+2^{\frac{ (k+3)(1+6c)}{1+c}}\varepsilon^{\frac{q+1}{q}} \Big),
 		\end{align}
 		where all the exponents of   scale invariant quantities are positive due to the facts $c\in \Big(\frac{5-q}{6q-5},\ \frac14\Big)$ and $q\in \Big(\frac52, \frac{10}{3}\Big)$.
 		
 		For the last two terms on the right-hand side of \eqref{3.11},  by   \eqref{1.3}, \eqref{3.12}, and H\"older's inequality, we  obtain
 		\begin{align}\label{3.17}
 			\int_{B_{k+1}}|u|^2\ dx\leq&N \rho_{k+1}^\frac53\Big( \int_{B_{k+1}}|u|^3\ dx\Big)^\frac23\nonumber\\
 			\leq& N \rho_{k+1}^{3}\Big(A(\rho_{k+1})^{\frac{q+3aq-10a}{6q}}C(\rho_{k+1})^{\frac{2a}{3q}}E(\rho_{k+1})^{\frac{5(q-aq+2a)}{6q}} \nonumber\\
 			& + A(\rho_{k+1})^\frac{3-a}{3}C(\rho_{k+1})^\frac{2a}{3q}\Big)\notag\\
 			\leq &N\Big(\varepsilon^{\frac{q+5aq-10a}{3q^2}}E(\rho_{k+1})^\frac{5(q-aq+2a)}{6q}+\varepsilon^\frac{2}{q}\Big)
 		\end{align}
 		and
 		\begin{align}\label{3.18}
 			\int_{B_{k+1}}|f||u|\ dx&\leq \Big(  \int_{B_{k+1}}|f|^2\ dx\Big)^\frac12 \Big(  \int_{B_{k+1}}|u|^2\ dx\Big)^\frac12\nonumber\\
 			& \leq   \rho_{k+1}A(\rho_{k+1})^\frac12F(\rho_{k+1})^\frac12\leq N\varepsilon^{\frac{q+2}{2q}},
 		\end{align}
 		where $N=N(a, q)>0$.
 		
 		The combination of   \eqref{3.11}, \eqref{3.12}, and  \eqref{3.16}-\eqref{3.18}   yields
 		\begin{align*}
 			E(\rho_k)\leq&    2^{2k}N\Big(\varepsilon^{ \frac{q+5aq-10a}{2q^2}}E(\rho_{k+2})^{\frac{5(q-aq+2a)}{4q}}+\varepsilon^\frac{3}{q}
 			+ \varepsilon^{\frac{q+2}{2q}}\\  &+\varepsilon^{\frac{2-3c}{q(1+c)}}E(\rho_{k+2})^\frac{1+6c}{2(1+c)}
 			+\varepsilon^{\frac{2}{q}}E(\rho_{k+2})^\frac12  +2^{\frac{ (k+3)(1+6c)}{1+c}}\varepsilon^{\frac{q+1}{q}}  \\
 			&+\varepsilon^{\frac{q+5aq-10a}{3q^2}}E(\rho_{k+2})^{\frac{5(q-aq+2a)}{6q}}+\varepsilon^\frac{2}{q}+\varepsilon^{\frac{q+2}{2q}}\Big)
 			\\
 			\leq &2^{2k}N\Big(\varepsilon^{ \frac{q+5aq-10a}{2q^2}}E(\rho_{k+2})^{\frac{5(q-aq+2a)}{4q}}
 			+\varepsilon^{\frac{2-3c}{q(1+c)}}E(\rho_{k+2})^\frac{1+6c}{2(1+c)} +\varepsilon^{\frac{2}{q}}E(\rho_{k+2})^\frac12  \\
 			&+\varepsilon^{\frac{q+5 aq-10a}{3q^2}}E(\rho_{k+2})^{\frac{5(q-aq+2a)}{6q}}+ 2^{\frac{ (k+3)(1+6c)}{1+c}} \varepsilon^{\frac{q+1}{q}} +\varepsilon^\frac{2}{q}\Big),
 		\end{align*}
 		where $N=N(a, c, q)$. Since $ c\in \Big(\frac{5-q}{6q-5}, \frac14\Big)$,  if we take  $a\in \Big(\frac{q}{5q-10}, 1\Big)$,   then all the  exponents of $E(\rho_{k+2})$ in the above inequality are  positive and  smaller than one.
 		Thus, by   Young's inequality, for any $\delta>0$, we have
 		\begin{align*}
 			E(\rho_k)\leq  \delta^2 E(\rho_{k+2})+2^{4k} N\varepsilon^\beta,
 		\end{align*}
 		where  $ N=N(c, q, \delta)>0$, $\beta=\beta(c, q)>0$ are some constants.
 		
 		Let $\delta= 3^{-4}$. Multiplying  both sides of the above inequality by $\delta^k$  and summing over $k$ from 1 to infinity, one obtains
 		\begin{align*}
 			\sum_{k=1}^\infty \delta^k E(\rho_k)\leq \sum_{k=3}^\infty  \delta^{k}E(\rho_k)+  N\varepsilon^\beta \sum_{k=1}^\infty\Big(\frac{2}{3}\Big)^{4k},
 		\end{align*}
 		which implies
 		\begin{align*}
 			E( 3/4)\leq  N\varepsilon^\beta.
 		\end{align*}
 		The lemma is proved.
 	\end{proof}
 	
 	Based on the above three lemmas, we will verify the decay estimates of  $A(x,\rho)+ E(x, \rho)$, $G(x, \rho)$, and $ P(x, \rho)$ with respect to  radius when $A(\rho_0), E(\rho_0), F(\rho_0)$, and  $P(\rho_0)$ are sufficiently small.
 	
 	\begin{lemma}\label{lem 3.4}
 		There exists a constant $\tilde\varepsilon>0$ satisfying the following property. If  we have
 		\begin{align}\label{3.20}
 			A(\rho_0)+E(\rho_0)+F(\rho_0)^\frac{2-l}{2(1+l)}+P(\rho_0)^
 			\frac{2-l}{1+l}  \leq \tilde\varepsilon,
 		\end{align}
 		then there exists sufficiently small constant  $l\in \Big(0,  \frac{4q-10}{5q-5}\Big)$ such that for any  $x\in B_{\rho_0/6}$ and   $\rho\in (0, \rho_0/2)$, the following inequalities hold
 		\begin{align}\label{3.21}
 			A(x, \rho)+ E(x, \rho)\leq N\tilde\varepsilon^{\frac{2-l}{2}}  \Big(\frac{\rho}{\rho_0} \Big)^{2-l},
 		\end{align}
 		\begin{align}\label{3.22}
 			G(x, \rho)\leq N \tilde\varepsilon^\frac{5(1+c)(1+l) }{2(4-c)} \Big(\frac{\rho}{\rho_0} \Big) ^{\frac{5(1+c)(1+l)}{4-c}},
 		\end{align}
 		and
 		\begin{align}\label{3.23}
 			P(x, \rho)\leq N\tilde\varepsilon^{\frac{1+l}{2}} \Big(\frac{\rho}{\rho_0} \Big)^{1+l},
 		\end{align}
 		where $N$ is a positive constant depending on $c, l$, but independent of $\tilde\varepsilon, \rho$ and $x$.
 	\end{lemma}
 	\begin{proof}
 		We use an induction argument.
 		Let  $\rho_k={\tilde\rho }^{(1+\beta)^k}$, where  $\tilde\rho=\rho_0/2$ and $\beta>0$ is some  small constant which will be specified later.  We fix an auxiliary parameter $\alpha\in (2-l, 2)$.  By a scaling   argument,  we  first   consider the special case: $\tilde\rho^{\alpha }=N \tilde\varepsilon $. For any $x\in B_{\rho_0/6}$, it is sufficient for us to prove the following decay estimates:
 		\begin{align}\label{3.24}
 			A(x, \rho_k)+ E(x, \rho_k)  \leq  \rho_k^{\alpha},   \quad  G(x, \rho_k)\leq \rho_k^{\frac{5\alpha(1+c)(1+l)}{(4-c)(2-l)}},\quad P(x, \rho_k)\leq \rho_k^{\frac{\alpha(1+l)}{2-l}} .
 		\end{align}

 		Let us first verify the  above estimates when $k=0$.
 		For  any  $x\in B_{\rho_0/6}$, since  $\tilde\rho=\rho_0/2$, by the   condition \eqref{3.20} and the fact $B(x,\tilde\rho)\subset  B_{\rho_0}$,
 		we get
 		\begin{align}\label{3.26}
 			A(x, \tilde\rho) + E(x, \tilde\rho) + F(x, \tilde\rho)^\frac{2-l}{2(1+l)}+ P(x, \tilde\rho)^\frac{2-l}{1+l}\leq N\tilde\varepsilon.
 		\end{align}
 		For $G(x, \tilde\rho)$,   since   $B(x, 5\tilde\rho/3)\subset B_{\rho_0}$, by \eqref{3.20}  and  Lemma \ref{lem 2}, we have
 		\begin{align}\label{3.29}
 			&G(x, \tilde\rho)\nonumber\\
 			&\leq  N\Big(P(x,5\tilde\rho/3)^{\frac{5(1+c)}{4-c}}+A(x, 5\tilde\rho/3)^\frac{5(1-4c)}{2(4-c)}E(x,5\tilde\rho/3)^\frac{5(1+6c)}{2(4-c)}+F(x, 5\tilde\rho/3)^\frac{5(1+c)}{2(4-c)}\Big)\nonumber\\
 			&\leq N\Big(P( \rho_0)^{\frac{5(1+c)}{4-c}}+A( \rho_0)^\frac{5(1-4c)}{2(4-c)}E( \rho_0)^\frac{5(1+6c)}{2(4-c)}+F( \rho_0)^\frac{5(1+c)}{2(4-c)}\Big)\nonumber\\
 			&\leq N\Big(\tilde\varepsilon^ {\frac{1+l}{2-l}\cdot\frac{5(1+c)}{4-c}}+\tilde\varepsilon^{ \frac{5(1-4c)}{2(4-c)}+ \frac{5(1+6c)}{2(4-c)}}+\tilde\varepsilon^{\frac{2(1+l)}{2-l}\cdot\frac{5(1+c)}{2(4-c)}}\Big)\leq( N \tilde\varepsilon)^{\frac{5(1+c)(1+l)}{(4-c)(2-l)}},
 		\end{align}
 		where $N=N(c, l)$.
 		Due to  \eqref{3.26}  and \eqref{3.29}, we can choose $N>0$ such that $\tilde\rho^\alpha=N\tilde\varepsilon $.  Hence, we obtain
 		\begin{align}\label{3.30}
 			A(x, \tilde\rho) \leq   \tilde\rho ^{\alpha},  \quad E(x, \tilde\rho) \leq \tilde\rho^{\alpha }, \quad  G(x, \tilde\rho) \leq \tilde\rho^{\frac{5\alpha(1+c)(1+l)}{(4-c)(2-l)}}, \quad
 			P(x, \tilde\rho) \leq \tilde\rho^{\frac{\alpha(1+l)}{2-l}}.
 		\end{align}
 		Thus,  we proved \eqref{3.24} for $k=0$ .
 		
 		Suppose  \eqref{3.24} holds for $0$ to $k$, we want to prove that it is also  true for $k+1$.
 		
 		We  first  estimate $A(x,\rho_{k+1})+E(x,\rho_{k+1})$. Taking $\gamma=\rho_k^{\beta}$ and  $\rho=\rho_k$ in  \eqref{3.1}, we have
 		\begin{align*}
 			&A(x,\rho_{k+1}) +E(x,\rho_{k+1})\\
 			&\leq \Big(\rho_k^{2\beta+\alpha}+\rho_k^{-2\beta+\frac{\alpha}{4}+\frac{5\alpha}{4}}+\rho_k^{-2\beta+\frac{3\alpha}{2}}+\rho_k^{-2\beta+\frac{\alpha(9c-1)}{4(1+c)}+\frac{\alpha(3-7c)}{4(1+c)}+\frac{5\alpha(1+c)(1+l)}{(4-c)(2-l)}\cdot\frac{4-c}{5(1+c)}}\\
 			&\quad +\rho_k^{-2\beta+\frac{\alpha}{2}+\frac{5\alpha(1+c)(1+l)}{(4-c)(2-l)}\cdot\frac{4-c}{5(1+c)}}+\rho_k^{6-\frac{10}{q}-4\beta}\|f\|_{L^q}^2\Big)\\
 			&\leq N\Big(\rho_k^{2\beta+\alpha}+\rho_k^{-2\beta+\frac{3}{2}\alpha} +\rho_k^{-2\beta+\frac{\alpha(4+l)}{4-2l}} +\rho_k^{6-\frac{10}{q}-4\beta}\|f\|_{L^q}^2\Big),
 		\end{align*}
 		where $N=N(c)$.
 		If we take $\beta< \min \big\{ \frac{\alpha}{2(\alpha+2)},  \frac{3\alpha l}{(4-2l)(\alpha+2)}, \frac{6q-\alpha q-10}{(4+\alpha)q}\big \}$,  then we have
 		\begin{align*}
 			\min\Big\{ 2\beta+\alpha,\  -2\beta+ \frac32\alpha, \ -2\beta+\frac {\alpha(4+l)}{4-2l}, \  6- \frac{10}{q}-4\beta  \Big\}> \alpha(1+\beta).
 		\end{align*}
 		Hence, there exists a small constant $\xi>0$ such that
 		\begin{align}\label{3.31}
 			A(x, \rho_{k+1})+E(x, \rho_{k+1})\leq N \rho_{k}^{\alpha(1+\beta)+\xi(1+\beta)}\leq N\rho_{k+1}^{\alpha+\xi}.
 		\end{align}
 		We choose a sufficiently small $\tilde\varepsilon$  which satisfies
 		\begin{align*}
 			N\rho_{k+1}^\xi< N\tilde\rho^\xi<N(N\tilde\varepsilon)^{\frac{\xi}{2}}\leq 1.
 		\end{align*}
 		Inserting  the above inequality into \eqref{3.31}, we have
 		\begin{align}\label{3.32}
 			A(x, \rho_{k+1}) +  E(x, \rho_{k+1}) \leq \rho_{k+1}^{\alpha}.
 		\end{align}
 		
 		Next we bound $G(x, \rho_{k+1})$. Taking $\gamma=\rho_k^{\beta}$ and  $\rho=\rho_k$ in \eqref{3.7}, by  \eqref{3.24} and H\"older's inequality,  one has
 		\begin{align}\label{3.33}
 			G(x, \rho_{k+1})\leq& N\Big(\rho_k^{\big(-\frac{10-15c}{4-c}+\frac{25}{4-c}\big)\beta+\frac{5\alpha(1+c)(1+l)}{(4-c)(2-l)}}+\rho_k^{-\frac{(10-15c)\beta}{4-c} +\frac{5\alpha(1-4 c)}{2(4-c)}+\frac{5\alpha(1+6c)}{2(4-c)}}\nonumber\\
 			&+\rho_k^{-\frac{5(2-3c)\beta}{4-c}+\frac{15(1+c)q-25(1+c)}{(4-c)q}}\|f\|_{L^q}^\frac{5(1+c)}{4-c}\Big)\nonumber\\
 			\leq& N\Big(\rho_k^{\frac{15(1+c)\beta}{4-c} +\frac{5\alpha(1+c)(1+l)}{(4-c)(2-l)}}+\rho_k^{-\frac{(10-15c)\beta}{4-c} +\frac{5\alpha(1+ c)}{4-c}}\nonumber\\
 			&+\rho_k^{-\frac{5(2-3c)\beta}{4-c}+\frac{5(1+c)(3q-5)}{(4-c)q}}\|f\|_{L^q}^\frac{5(1+c)}{4-c}\Big),
 		\end{align}
 		where $N=N(c)$.
 		If we take
 		\begin{align*}
 			0<\beta <
 			\frac{(1+c)[6q-\alpha q-10-(3q+\alpha q-5)l]}{q[(1+c)(1+l)\alpha+(2-3c)(2-l)]}
 		\end{align*}
 		and
 		\begin{align*}
 			0<l < \frac{4q-10}{5q-5}<  \frac{6q-\alpha q-10}{3q+\alpha q-5},
 		\end{align*}
 		then  all the exponents of $\rho_k$ on the right-hand side of \eqref{3.33} are greater than $\frac{5\alpha(1+c)(1+l)(1+\beta)}{(4-c)(2-l)}$.
 		Hence, there exists a small constant $\xi>0$ such that \begin{align}\label{3.36}
 			G(x, \rho_{k+1})\leq N\rho_{k}^{\big[\frac{5\alpha(1+c)(1+l)}{(4-c)(2-l)}+\xi\big](1+\beta)}\leq N\rho_{k+1}^{\frac{5\alpha(1+c)(1+l)}{(4-c)(2-l)}+\xi}.
 		\end{align}
 		By choosing  $\tilde\varepsilon$ sufficiently small,  one has
 		\begin{align*}
 			N\rho_{k+1}^{\xi}< N \tilde\rho^\xi\leq N(N\tilde\varepsilon)^\frac{\xi}{2}<1.
 		\end{align*}
 		Inserting the  above inequality  into \eqref{3.36}, we have
 		\begin{align}\label{3.37}
 			G(x, \rho_{k+1})\leq \rho_{k+1}^{\frac{5\alpha(1+c)(1+l)}{(4-c)(2-l)}}.
 		\end{align}
 		
 		For  $P(x,\rho_{k+1})$, by H\"older's inequality, we have
 		\begin{align}\label{p}
 			&\int_{B(x,\rho_{k+1})} |p-(p)_{B(x,\rho_{k+1})}|\ dx\notag\\
 			&\leq N \rho_{k+1}^\frac{1+6c}{1+c}\Big(\int_{B(x,\rho_{k+1})}
 			|p-(p)_{B(x,\rho_{k+1})}|^\frac{5(1+c)}{4-c} \ dx\Big)^\frac{4-c}{5(1+c)}.
 		\end{align}
 		The combination of \eqref{3.36}  and \eqref{p}  implies
 		\begin{align}\label{3.38}
 			P(x, \rho_{k+1})\leq N G(x, \rho_{k+1})^\frac{4-c}{5(1+c)}\leq N \rho_{k+1}^{\frac{\alpha(1+l)}{2-l}+\frac{(4-c)\xi}{5(1+c)}},
 		\end{align}
 		where $N=N(c)$.
 		We choose  $\tilde\varepsilon$ sufficiently small  such that
 		\begin{align*}
 			N\rho_{k+1}^{\frac{(4-c)\xi}{5(1+c)}}<N\tilde\rho^{\frac{(4-c)\xi}{5(1+c)}}\leq N (N\tilde\varepsilon)^\frac{(4-c)\xi}{10(1+c)}<1.
 		\end{align*}
 		Inserting the  above inequality into \eqref{3.38}, we have
 		\begin{align}\label{3.39}
 			P(x, \rho_{k+1})\leq \rho_{k+1}^\frac{\alpha(1+l)}{2-l}.
 		\end{align}
 		
 		Using \eqref{3.30}, \eqref{3.32}, \eqref{3.37}, and \eqref{3.39}, by induction we obtain \eqref{3.24}  for any integer $k\geq 0$.
 		
 		For any  $\rho\in \big(0, \rho_0/2\big)$, there exists a positive integer $k$ such that $\rho_{k+1}\leq\rho<\rho_k$.  By \eqref{3.24}, we obtain
 		\begin{align}\label{3.40}
 			A(x, \rho)\leq \Big(\frac{\rho_k}{\rho_{k+1}}\Big)^{3}A(x, \rho_k)\leq \rho_k^{\alpha-3\beta}=\rho_{k+1}^{\frac{\alpha-3\beta}{1+\beta}}\leq \rho^{2-l},
 		\end{align}
 		\begin{align*}
 			E(x, \rho)\leq \frac{\rho_k}{\rho_{k+1}}E(x, \rho_k)\leq \rho_k^{\alpha-\beta}=\rho_{k+1}^{\frac{\alpha-\beta}{1+\beta}}\leq \rho^{2-l},
 		\end{align*}
 		\begin{align}\label{3.42}
 			G(x, \rho)\leq &\Big(\frac{\rho_k}{\rho_{k+1}}\Big)^{\frac{10-15c}{4-c}}G(x, \rho_k)     \leq \rho_k^{\frac{5\alpha(1+c)(1+l)}{(4-c)(2-l)}-\frac{(10-15c)\beta}{4-c}}\nonumber\\
 			=&\rho_{k+1}^{ \frac{5\alpha(1+c)(1+l)}{(4-c)(1+\beta)(2-l)}-\frac{(10-15c)\beta}{(4-c)(1+\beta)}}
 			\leq \rho^{\frac{5(1+c)(1+l)}{4-c}},
 		\end{align}
 		and
 		\begin{align}\label{3.43}
 			P(x, \rho)\leq \Big(\frac{\rho_k}{\rho_{k+1}}\Big)^{3}P(x, \rho_k)\leq \rho_k^{\frac{\alpha(1+l)}{2-l}-3\beta}=\rho_{k+1}^{\frac{1}{1+\beta}\big[\frac{\alpha(1+l)}{2-l}-3\beta\big]}\leq \rho^{1+l},
 		\end{align}
 		where we take $\beta$ sufficiently small such that  $\beta\in \Big(0, \frac{(1+l)(\alpha-2+l)}{(2-l)(4+l)}\Big)$.

 		Hence, the lemma is proved when $\tilde\rho^{\alpha}=N\tilde\varepsilon $.
 		
 		Next we consider the  general case. Recall \eqref{eq10.23}. By taking $\lambda=\rho_0/(N\tilde\varepsilon)^{\frac{1}{\alpha}}$,
 		we know that $(u_\lambda, p_\lambda)$ is also a suitable weak solution to \eqref{1.1} in $  B\big(x, (N\tilde\varepsilon)^\frac{1}{\alpha}\big)$. Moreover,
 		\begin{align*}
 			A(x, \rho)=\frac{1}{\rho^3}\int_{B(x, \rho)}|u(y)|^2dy=\Big(\frac{\lambda}{\rho}\Big)^3\int_{B(x, \rho/\lambda)}|u_\lambda(y)|^2dy,
 		\end{align*}
 		\begin{align*}
 			E(x, \rho)=\frac{1}{\rho}\int_{B(x, \rho)}|\nabla u(y)|^2 dy=\frac{\lambda}{\rho}\int_{B(x, \rho/ \lambda)}|\nabla u_\lambda (y)|^2 dy,
 		\end{align*}
 		\begin{align*}
 			G(x, \rho)&=\frac{1}{\rho^{\frac{10-15c}{4-c}}}\int_{B(x, \rho)}|p(y)-(p)_{B(x, \rho)}|^\frac{5(1+c)}{4-c}dy\nonumber\\
 			&=\Big(\frac{\lambda}{ \rho}\Big)^{\frac{10-15c}{4-c}}\int_{B(x, \rho/\lambda)}|p_\lambda(y)-(p_\lambda)_{B(x,
 				\rho/\lambda)}|^\frac{5(1+c)}{4-c}dy
 		\end{align*}
 		as well as
 		\begin{align*}
 			P(x, \rho)&=\frac{1}{\rho^3}\int_{B(x, \rho)}|p(y)-(p)_{B(x,\rho)}|dy\nonumber\\
 			&=\Big(\frac{\lambda}{\rho}\Big)^3\int_{B(x, \rho/\lambda)}|p_\lambda(y)-(p_\lambda)_{B(x, \rho/\lambda)}|^2 dy.
 		\end{align*}
 		Applying \eqref{3.40}-\eqref{3.43} to $(u_\lambda,p_\lambda)$, we get
 		\begin{align*}
 			A(x, \rho)+E(x, \rho) \leq \Big(\frac{\rho}{\lambda}\Big)^{2-l}\leq N\tilde\varepsilon^\frac{2-l}{\alpha}\Big(\frac{\rho}{\rho_0}\Big)^{2-l}\leq N\tilde\varepsilon^\frac{2-l}{2}\Big(\frac{\rho}{\rho_0}\Big)^{2-l},
 		\end{align*}
 		\begin{align*}
 			G(x, \rho)\leq \Big(\frac{\rho}{\lambda}\Big)^\frac{5(1+c)(1+l)}{4-c} \leq N\tilde\varepsilon^\frac{5(1+c)(1+l)}{2(4-c)}\Big(\frac{\rho}{\rho_0}\Big)^\frac{5(1+c)(1+l)}{4-c}
 		\end{align*}
 		as well as
 		\begin{align*}
 			P(x, \rho) \leq  \Big(\frac{\rho}{\lambda}\Big)^{1+l
 			}\leq N\tilde\varepsilon^\frac{1+l}{2}\Big(\frac{\rho}{\rho_0}\Big)^{1+l
 			},
 		\end{align*}
 		where $N=N(c, l)$.
 		Thus, the lemma is proved.
 	\end{proof}
 	
 	In the rest of this section, we will  utilize   Lemma \ref{lem 3.4} to prove Theorem \ref{Them 1}.
 	Let   $\rho_0=3/4$. By the  condition \eqref{1.3}, H\"older's inequality which gives
 	\begin{align*}
 		A(1) \leq NC(1)^\frac{2}{q}\leq N\varepsilon^\frac{2}{q}
 	\end{align*}
 	and   Lemma \ref{lem 3.3}, we can choose  $\varepsilon>0$ sufficiently small such that the condition \eqref{3.20} holds. Hence, we  obtain  the decay estimates  \eqref{3.21}-\eqref{3.23} in  Lemma \ref{lem 3.4}. Let  $x\in B_{1/8}$ and   $\rho\in\big (0,3/8\big)$. We  decompose the suitable  weak solution $u$ of \eqref{1.1} as  $u=w+v$, where $w$ satisfies the equation
 	\begin{align*}
 		\Delta w_i=\partial_i\big(p-(p)_{B(x, \rho)}\big)+\partial_j\big(u_iu_j\big)+f_i
 	\end{align*}
 	with the zero boundary condition.  Hence, by the Calder\'on-Zygmund  estimate, we have
 	\begin{align}\label{3.48}
 		&\big\|\nabla w\big\|_{L^\frac{10}{7}(B(x,\rho))}\nonumber\\
 		&\leq N\Big(\big\| p-(p)_{B(x, \rho)}\big\|_{L^\frac{10}{7}(B(x,\rho))}+\big\||u|^2\big\|_{L^\frac{10}{7}(B(x,\rho))}+\rho\big\|f\big\|_{L^\frac{10}{7}(B(x,\rho))}\Big).
 	\end{align}

 	
 	We  first estimate the pressure term on the right-hand side of   \eqref{3.48}.  By  \eqref{3.22} and \eqref{3.23} in Lemma \ref{lem 3.4}, we have
 	\begin{align*}
 		G(x, \rho)\leq N  \rho^{\frac{5(1+c)(1+l)}{4-c}},\quad P(x,\rho)\leq N \rho^{1+l},
 	\end{align*}
 	where $N=N(c, l)>0$ is some constant.
 	Thus, by  H\"older's inequality and the  above estimates,  one derives
 	\begin{align}\label{3.49}
 		&\Big(\int_{B(x,\rho)}|p-(p)_{B(x, \rho)}|^\frac{10}{7}\ dx\Big)^\frac{7}{5}\nonumber\\
 		&\leq  \Big(\int_{B(x,\rho)}|p-(p)_{B(x, \rho)}| \ dx\Big)^\frac{9c-1}{1+6c}\Big(\int_{B(x,\rho)}|p-(p)_{B(x, \rho)}|^\frac{5(1+c)}{4-c}\ dx\Big)^\frac{3(4-c)}{5(1+6c)}\nonumber\\
 		&\leq \rho^3P(x,\rho)^\frac{9c-1}{1+6c}G(x,\rho)^\frac{3(4-c)}{5(1+6c)}\leq  N\rho^{5+2l}.
 	\end{align}
 	
 	For the second term on the right-hand side of \eqref{3.48},  due to  \eqref{3.21}   in  Lemma \ref{lem 3.4}, we have
 	\begin{align*}
 		A(x,\rho)+ E(x, \rho)  \leq N \rho^{2-l},
 	\end{align*}
 	where $N=N(c, l)$.
 	Thus,  by  the Sobolev embedding inequality and  the above decay rate, one obtains
 	\begin{align}\label{3.50}
 		&\Big( \int_{B(x,\rho)}|u|^\frac{20}{7} \ dx\Big)^\frac{7}{5}\nonumber\\
 		&\leq N  \Big(\int_{B(x,\rho)}|u|^2 \ dx\Big)^\frac12\Big(\int_{B(x,\rho)}|\nabla u|^2 \ dx+\rho^{-2}\int_{B(x,\rho)}|u|^2 \ dx\Big)^\frac32\nonumber\\
 		&\leq N \rho^3  A(x,\rho)^\frac12\big(E(x,\rho)+A(x,\rho)\big)^\frac32\leq N \rho^{7-2l},
 	\end{align}
 	where $N=N(c, l)$.
 	
 	For  the last  term on the right-hand side of \eqref{3.48},  by H\"older's
 	inequality, one  derives
 	\begin{align}\label{3.51}
 		\rho^2 \Big(\int_{B(x,\rho)}|f|^\frac{10}{7} \ dx\Big)^\frac{7}{5}&\leq N \rho^{9-\frac{10}{q}}\Big(\int_{B(x,\rho)}| f|^q \ dx\Big)^\frac{2}{q}\nonumber\\
 		&= N \rho^{9-\frac{10}{q}} \| f\|_{L^q(B_1)}^2.
 	\end{align}

 	The combination of the Sobolev-Poincar\'e  inequality and  \eqref{3.48}-\eqref{3.51} yields
 	\begin{align}\label{3.52}
 		&\int_{B(x, \rho)}|w-(w)_{B(x, \rho)}|^2 \ dx\leq  N\Big(\int_{B(x, \rho)}|\nabla  w|^\frac{10}{7} \ dx\Big)^\frac{7}{5}\nonumber\\
 		&\leq N\Big[  \Big(\int_{B(x,\rho)}|p-(p)_{B(x, \rho)}|^\frac{10}{7}\ dx\Big)^\frac{7}{5}+\Big( \int_{B(x,\rho)}|u|^\frac{20}{7} \ dx\Big)^\frac{7}{5}\nonumber\\
 		&\quad + \rho^2 \Big(\int_{B(x,\rho)}|f|^\frac{10}{7} \ dx\Big)^\frac{7}{5}\Big]
 		\nonumber\\
 		&\leq N\Big(  \rho^{ 5+2l}+ \rho^{7-2l}+\rho^{9-\frac{10}{q}}\|f\|_{L^q(B_1)}^2\Big),
 	\end{align}
 	where $N=N(c, l)$.
 	
 	Since any Sobolev norm of harmonic function $v-(v)_{B(x, \gamma\rho)}$ in $B(x, \gamma\rho)$ can be controlled by the $L^p$ norm of it in $B(x, \rho)$ for any  $p\in [1, +\infty]$,  we have
 	\begin{align}\label{3.53}
 		\int_{B(x,\gamma\rho)}|v-(v)_{B(x, \gamma\rho)}|^2 \ dx \leq & N (\gamma\rho)^2\int_{B(x,\gamma\rho)}|\nabla v|^2 \ dx\nonumber\\
 		\leq & N (\gamma\rho)^7\sup_{B(x,\gamma\rho)}|\nabla v|^2\nonumber\\
 		\leq & N\gamma^{7} \int_{B(x,\rho)}|v-(v)_{B(x, \rho)}|^2\ dx,
 	\end{align}
 	where $\gamma\in (0, 1)$ is any constant.
 	
 	By \eqref{3.52} and \eqref{3.53}, one derives
 	\begin{align}\label{3.54}
 		&\int_{B(x,\gamma\rho)}|u-(u)_{B(x, \gamma)\rho}|^2 \ dx\nonumber\\
 		&\leq \int_{B(x,\gamma\rho)}|v-(v)_{B(x, \gamma\rho)}|^2 \ dx+\int_{B(x,\gamma\rho)}|w-(w)_{B(x,\gamma\rho)}|^2 \ dx\nonumber \\
 		&\leq N\Big(\gamma^{7} \int_{B(x,\rho)}|v-(v)_{B(x, \rho)}|^2\ dx+ \rho^{ 5+2l}+  \rho^{7-2l}+\rho^{9-\frac{10}{q}}\|f\|_{L^q(B_1)}^2\Big)\nonumber \\
 		&\leq N\Big(\gamma^{7} \int_{B(x,\rho)}|u-(u)_{B(x, \rho)}|^2\ dx+  \rho^{ 5+2l}+  \rho^{7-2l}+\rho^{9-\frac{10}{q}}\|f\|_{L^q(B_1)}^2\Big),
 	\end{align}
 	where $N=N(c, l)$.
 	
 	Due to  the condition of $l$ in Lemma \ref{lem 3.4}, by  taking a sufficiently small  constant $ l \in \Big(0,  \frac{4q-10}{ 5q-5}\Big)  $,  we have
 	\begin{align*}
 		\min\big\{ 5+2l, 7-2l, 9- 10/q\big\}>5+l.
 	\end{align*}
 	The combination of \eqref{3.54} and Lemma \ref{lem 4.5} implies that
 	\begin{align*}
 		\int_{B(x, \rho)}|u-(u)_{B(x, \rho)}|^2 \ dx\leq N \rho^{5+l}
 	\end{align*}
 	for any  $x\in B_{1/8}$ and $\rho\in \big(0, 3/8\big)$. Hence, by Campanato's characterization of H\"older continuity, $u$ is H\"older continuous in $B_{1/2}$ .
 	
 	\section{Boundary  \texorpdfstring{$\varepsilon$}{}-regularity }
 	In this section, we give the proof Theorem \ref{Them 2}. We first prove the smallness of  $E^+(15/16)$ in Lemma \ref{lem 4.3} by using Lemmas \ref{lem 4.1} and \ref{lem 4.2} and an iteration argument. Then, based on the above results and the condition \eqref{1.4b},  we establish the uniform decay estimates of $A^++E^+$, $G^+,$  and $P^+$ in Lemma \ref{lem 4.4}. Finally, by the  decay estimates of scale invariant quantities and the $L^p$ estimate for elliptic equations, we obtain the H\"older continuity of $u$ by Campanato's characterization of H\"older continuity. Throughout this section, we use  the pair $(u, p)$ to represent a suitable weak solution to the incompressible Navier-Stokes equations \eqref{1.1} with the boundary condition \eqref{1.2}.
 	
 	First, as in the interior case, by using \eqref{1.4b} and \cite[Theorem 3.8]{K}, without loss of generality, we may assume that
 	\begin{align}\label{1.4}
 		\int_{B_1^+}|u|^q\ dx+\int_{B_1^+}|p-(p)_{B_1^+}|\ dx + \int_{B_1^+}|f|^{2}\ dx <\varepsilon,
 	\end{align}
 	
 	In the following  lemmas,  we prove that  the values of $A^++E^+ $ and $ G^+ $ in a smaller half ball can be controlled by their values in a larger half ball.
 	
 	\begin{lemma}\label{lem 4.1}
 		Let $\Omega=B_1^+$. For any $\gamma \in (0,1/2]$,  $x_0\in \partial \Omega\cap\{x_5=0\}$ and $B^+(x_0, \rho)\subset B^+_1$, we have
 		\begin{align*}
 			&A^+(x_0, \gamma\rho)+E^+(x_0, \gamma\rho)\nonumber\\
 			&\leq N\Big[\gamma^2 E^+(x_0, \rho) +\gamma^{-2}E^+(x_0, \rho)^\frac32+ \gamma^{-2}    E^+(x_0, \rho)^\frac 12G^+(x_0, \rho)^\frac{4-c}{5(1+c)}\notag\\
 			&\quad +\gamma^{-4} F^+(x_0, \rho)\Big],
 		\end{align*}
 		where  $N=N(c)$ is a positive constant independent of $\gamma$ and $\rho$.
 	\end{lemma}
 	\begin{proof}
 		By the scale invariant property, we assume $\rho=1$.
 		Similar to the proof of Lemma \ref{lem 3.1}, we   have
 		\begin{align}\label{AE}
 			&A^+(x_0, \gamma)+E^+(x_0, \gamma)\nonumber\\
 			&\leq N\Big[\gamma^2 A^+(x_0, 1)+\gamma^{-2}A^+(x_0, 1)^\frac14E^+(x_0, 1)^\frac{5}{4}+\gamma^{-2}A^+(x_0, 1)^\frac32\notag\\
 			&\quad + \gamma^{-2} A^+(x_0, 1)^\frac{9c-1}{4(1+c)}   E^+(x_0, 1)^\frac{3-7c}{4(1+c)}G^+(x_0, 1)^\frac{4-c}{5(1+c)}\notag\\
 			&\quad +\gamma^{-2}A^+(x_0, 1)^\frac12G^+(x_0, \rho)^\frac{4-c}{5(1+c)}+\gamma^{-4} F^+(x_0, 1)\Big].
 		\end{align}
 		Inserting the boundary Poincar\'e inequality
 		\begin{align}\label{BP}
 			A^+(x_0, 1)\leq N E^+(x_0, 1)
 		\end{align}
 		into \eqref{AE}, we obtain the lemma.
 	\end{proof}
 	\begin{lemma}\label{lem 4.2}
 		Let $\Omega=B_1^+.$   For any  $\gamma\in (0,3/4]$, $x_0\in \partial \Omega\cap\{x_5=0\}$ and $B^+(x_0, \rho)\subset B_1^+$, we have
 		\begin{align*}
 			&G^+(x_0, \gamma\rho)\leq N \gamma^{-\frac{5(2-3c)}{4-c}}\Big[ E^+(x_0, \rho)^{\frac{5(1+c)}{4-c}}
 			+ F^+(x_0, \rho)^{\frac{5(1+c)}{2(4-c)}}\\
 			&\quad + \gamma^{\frac{25}{4-c}-\frac{25(1+c)}{r'(4-c)}}\Big( E^+(x_0, \rho)^{\frac{5(1+c)}{2(4-c)}} +P^+(x_0, \rho)^{\frac{5(1+c)}{4-c}} \Big)\Big],
 		\end{align*}
 		where $r'\in (1+c, +\infty)$ is any sufficiently large  constant and $N=N(c, r')$ is some positive constant.
 	\end{lemma}
 	\begin{proof}
 		By  the scale invariant property, we may suppose $\rho=1$ and  $x_0=(0,\dots, 0)$.  Let    $\tilde B\subset \mathbb{R}^5$ be a fixed domain with smooth boundary  which satisfies
 		\begin{align*}
 			B^+_{3/4}\subset \tilde B\subset B^+_1.
 		\end{align*}
 		Similar to  \eqref{3.15}, by H\"older's inequality, we have
 		\begin{align}\label{4.1}
 			\| u\cdot\nabla u\|_{L^{1+c}(B_1^+)}\leq&  N \Big(\|u\|_{L^2(B_1^+)}^{\frac{1-4c}{1+c}}\|\nabla u\|_{L^2(B_1^+)}^{\frac{1
 					+6c}{1+c}}+\|u\|_{L^2(B_1^+)}\|\nabla u\|_{L^2(B_1^+)}\Big)
 		\end{align}
 		and
 		\begin{align*}
 			\| f\|_{L^{1+c}(B_1^+)}\leq N\| f\|_{L^2(B_1^+)},
 		\end{align*}
 		where $N>0$ is some constant.
 		
 		Using  Theorem 6.1 of \cite{Ga94}, there exists a unique solution $(v, p_1)$, which satisfies the following equations
 		\begin{align*}
 			\begin{cases}
 				-\Delta v+\nabla p_1=-u\cdot\nabla u +f \quad \mathrm{in}\ \tilde B,\\
 				\nabla \cdot v=0 \quad \mathrm{in}\  \tilde B,\\
 				(p_1)_{\tilde B}=0 ,\\
 				v=0 \quad \mathrm{on}\ \partial  \tilde B.\\
 			\end{cases}
 		\end{align*}
 		Moreover, we have the following estimate
 		\begin{align}\label{4.3}
 			&\|v \|_{L^{1+c}(\tilde B)}+    \|\nabla  v \|_{L^{1+c}(\tilde B)}+\|p_1\|_{L^{1+c}(\tilde B)}+\|\nabla p_1\|_{L^{1+c}(\tilde B)}\nonumber\\
 			&\leq N\Big( \| u\cdot\nabla u\|_{L^{1+c}(\tilde B)}+ \| f\|_{L^{1+c} (\tilde B)}\Big)\nonumber\\
 			&\leq N\Big(\|u\|_{L^2(B_1^+)}^{\frac{1-4c}{1+c}}\|\nabla u\|_{L^2(B_1^+)}^{\frac{1
 					+6c}{1+c}} +\|u\|_{L^2(B_1^+)}\|\nabla u\|_{L^2(B_1^+)}+  \| f\|_{L^2(B_1^+)}\Big),
 		\end{align}
 		where we use  H\"older's inequality and the Sobolev embedding inequality in the last inequality,  $N=N(c)$.
 		
 		Let $w=u-v$ and $p_2=p-p_1$. It is easy  to see  $(w ,p_2)$ satisfies
 		\begin{align*}
 			\begin{cases}
 				-\Delta w+\nabla p_2=0\\
 				\nabla \cdot w=0
 			\end{cases}
 			\quad \mathrm{in}\ \tilde B
 		\end{align*}
 		and the boundary condition
 		\begin{align*}
 			w=0 \quad \mathrm{on}\  \{x\ : \ x=(x^1,\dots, x^4, 0)\ \}\cap \partial\tilde B.
 		\end{align*}
 		By \eqref{4.3} and the stationary case of the estimate of pressure in Theorem 1.2 of \cite{Se2}  (see also \cite{K}), one has
 		\begin{align}\label{4.4}
 			&\|\nabla p_2\|_{L^{1+c} (B_{\gamma}^+)}\leq N \gamma^{\frac{5}{1+c}-\frac{5}{r'}}\|\nabla p_2\|_{L^{r'} (B_{\gamma}^+)}\nonumber\\
 			&\leq N\gamma^{\frac{5}{1+c}-\frac{5}{r'}}\Big(\|w\|_{ L^{1+c}(B^+ _{ 3/4})}+\|\nabla w\|_{L^{1+c}( B^+ _{3/4})}+\|p_2-(p)_{B^+_1}\|_{L^{1}(B _{ 3/4}^+)}\Big)\nonumber\\
 			&\leq N\gamma^{\frac{5}{1+c}-\frac{5}{r'}}( \|u\|_{L^{1+c}(B^+ _{ 3/4})}+\|v\|_{L^{1+c}( B ^+_{3/4})}+\|\nabla u\|_{L^{1+c}(B _{ 3/4}^+)} \nonumber\\
 			&\quad +\| \nabla v \|_{L^{1+c}(B _{ 3/4}^+)} +\| p-(p)_{B^+_1}\|_{L^{1}(B _{ 3/4}^+)}+\|p_1\|_{L^1(B_{ 3/4}^+)}\Big)\nonumber\\
 			&\leq N\gamma^{\frac{5}{1+c}-\frac{5}{r'}}\Big(\|u\|_{L^{2}(B_1^+)}+  \|\nabla u\|_{L^2(B_1^+)}+\| p-(p) _{B^+_1}\|_{L^{1}(B_1^+)} \nonumber\\
 			&\quad +  \|u\|_{L^2(B_1^+)}^{\frac{1-4c}{1+c}}\|\nabla u\|_{L^2(B_1^+)}^{\frac{1
 					+6c}{1+c}}+\|u\|_{L^2(B_1^+)}\|\nabla u\|_{L^2(B_1^+)}+\ \| f\|_{L^2(B_1^+)} \Big),
 		\end{align}
 		where  $r'\in (1+c, +\infty)$ and $N=N(c, r')$ are  some positive constants.
 		
 		Combining  \eqref{4.3} and \eqref{4.4}, by the Sobolev-Poincar\'e  inequality, we have
 		\begin{align*}
 			\big\|  p&-(p)_{B^+_\gamma}\big\|_{L^\frac{5(1+c)}{4-c}(B_\gamma^+ )}\\
 			\leq& N\Big[\|\nabla p_1\|_{L^{1+c}(B_\gamma^+)}+\|\nabla p_2\|_{L^{1+c}(B_\gamma^+)}\Big]\\
 			\leq& N\Big[ \|u\|_{L^2(B_1^+)}^{\frac{1-4c}{1+c}} \|\nabla u\|_{L^2(B_1^+)}^{\frac{1
 					+6c}{1+c}}+\|u\|_{L^2(B_1^+)}\|\nabla u\|_{L^2(B_1^+)}+  \| f\|_{L^2(B_1^+)}\\
 			&+ \gamma^{\frac{5}{1+c}-\frac{5}{r'}} \Big(\|u\|_{L^{2}(B_1^+)}+ \|\nabla u\|_{L^2(B_1^+)} +\| p-(p)_{B_1^+}\|_{ L^ 1(B_1^+)}\Big)  \Big],
 		\end{align*}
 		where $N=N(c, r')$.
 		Hence, we get
 		\begin{align*}
 			& G^+(x_0, \gamma)\leq N \gamma^{-\frac{5(2-3c)}{4-c}}\Big[A^+(x_0, 1)^{\frac{5(1-4c)}{2(4-c)}}E^+(x_0, 1)^{\frac{5(1+6c)}{2(4-c)}}\\
 			&\quad +A^+(x_0, 1)^\frac{5(1+c)}{2(4-c)}E^+(x_0, 1)^\frac{5(1+c)}{2(4-c)}+ F^+(x_0, 1)^{\frac{5(1+c)}{2(4-c)}}\\
 			&\quad + \gamma^{\frac{25}{4-c}-\frac{25(1+c)}{r'(4-c)}}\Big(A^+(x_0, 1)^{\frac{5(1+c)}{2(4-c)}}+E^+(x_0, 1)^{\frac{5(1+c)}{2(4-c)}} +P^+(x_0, 1)^{\frac{5(1+c)}{4-c}} \Big)\Big].
 		\end{align*}
 		Inserting the boundary Poincar\'e inequality  \eqref{BP} into  the above inequality,
 		the conclusion of Lemma \ref{lem 4.2} follows immediately.
 	\end{proof}
 	
 	By  Lemmas \ref{lem 4.1} and \ref{lem 4.2},  we show below that $E^+(15/16)$ can be sufficiently small   under the condition  \eqref{1.4}. This lemma is needed for us to prove the decay estimates of   scale invariant quantities with respect to the  radius.
 	\begin{lemma}\label{lem 4.3}
 		Let $\Omega=B_1^+$ and $(u,p)$ be a suitable weak solution to \eqref{1.1}  satisfying the condition of Theorem \ref{Them 2}. Then  we have
 		\begin{align*}
 			E^+(15/16)\leq   N\varepsilon^\beta,
 		\end{align*}
 		where $   N, \beta$ are some positive  constants which depend only on $c$ and $q$.
 	\end{lemma}
 	\begin{proof}
 		Let    $\rho_k=1-2^{-4k}$  and $B_k^+=B^+_{\rho_k}$, where $k$ is any positive integer. We choose the domains $\tilde B_k$ satisfying
 		\begin{align*}
 			B_{k+1}^+\subset\tilde B_k\subset B_{k+2}^+.
 		\end{align*}
 		
 		For each $k$, we choose  cut-off functions $\psi_k$ such that
 		\begin{align*}
 			\mathrm{supp}\ \psi_k\subset B_{k+1}, \quad \psi_k=1\quad \mathrm{in} \ B_k,\\
 			| D\psi_k|\leq N 2^k, \quad |D^2\psi_k|\leq N 2^{2k}\quad  \mathrm{in}\ B_{k+1}^+.
 		\end{align*}
 		
 		By the energy inequality \eqref{2.1},  we have
 		\begin{align} \label{4.5}
 			&2\int_{B_k^+}|\nabla u|^2  \ dx\notag\\
 			&\leq N \int_{B_{k+1}^+}\Big(2^k|u|^3+2^{k+1}|p-(p)_{B^+_{k+1}}||u|  +2^{2k}|u|^2+2|f||u| \Big) \ dx.
 		\end{align}
 		
 		Next we estimate each term on the right-hand side of the above inequality.
 		
 		By    \eqref{1.4} and H\"older's inequality, we have
 		\begin{align}\label{AC2}
 			A^+(\rho_{k+1})\leq N C^+(\rho_{k+1})^\frac{2}{q}\leq NC^+(1)^\frac{2}{q}\leq N \varepsilon^\frac{2}{q}.
 		\end{align}

 		For the first  term on the right-hand side of \eqref{4.5},
 		similar to  \eqref{3.12}, by  H\"older's inequality, the Sobolev embedding  inequality and \eqref{AC2},  we  derive that
 		\begin{align}\label{4.6}
 			\int_{B_{k+1}^+}|u|^3\ dx\leq& N\Big(A^+(\rho_{k+1})^\frac{q+3aq-10a}{4q}C^+(\rho_{k+1})^\frac{a}{q}E^+(\rho_{k+1})^\frac{5(q-aq+2a)}{4q}\nonumber\\
 			&+A^+(\rho_{k+1})^\frac{3-a}{2}C^+(\rho_{k+1})^\frac{a}{q}\Big)\nonumber\\
 			\leq& N\Big(\varepsilon^{\frac{q+5aq-10a}{2q^2}}E^+(\rho_{k+1}
 			)^\frac{5(q-aq+2a)}{4q}+\varepsilon^{\frac{3}{q}}\Big),
 		\end{align}
 		where $N=N(a, q)$. Here, we take the constant  $a\in (0,1)$ to make sure all the exponents in  the  above  inequality are positive.
 		
 		For the third term on the right-hand side of \eqref{4.5}, similar  to    \eqref{3.17} in  Lemma \ref{lem 3.3}, we  have
 		\begin{align}\label{4.7}
 			\int_{B_{k+1}^+}|u|^2\ dx \leq& N  \Big(A^+(\rho_{k+1})^{\frac{q+3aq-10a}{6q}}C^+(\rho_{k+1})^{\frac{2a}{3q}}E^+(\rho_{k+1})^{\frac{5(q-aq+2a)}{6q}}\nonumber\\
 			&+ A^+(\rho_{k+1})^\frac{3-a}{3}C^+(\rho_{k+1})^\frac{2a}{3q}\Big)\nonumber\\
 			\leq & N \Big(\varepsilon^{\frac{q+5aq-10a}{3q^2}}E^+(\rho_{k+1})^\frac{5(q-aq+2a)}{6q}+\varepsilon^{\frac{2}{q}}\Big),
 		\end{align}
 		where $N=N(a, q)$.

 		For the last term on the right-hand side of \eqref{4.5},  by H\"older's inequality, we have
 		\begin{align}\label{4.8}
 			\int_{B_{k+1}^+}|f||u| \ dx\leq N   A^+(\rho_{k+1})^\frac12F^+(\rho_{k+1})^\frac12\leq N  \varepsilon^{\frac12+\frac{1}{q}}.
 		\end{align}

 		Next we  bound the second term on the right-hand side of \eqref{4.5}. Analogous to \eqref{3.13}, we have
 		\begin{align}\label{4.9}
 			\int_{B_{k+1}^+} |p-(p)_{B^+_{k+1}}|  |u| \ dx\leq& \Big(\int_{B_{k+1}^+} |\nabla p  |^{1+c} \ dx\Big)^\frac{1}{1+c}\Big(\int_{B_{k+1}^+}|u|^2 \ dx\Big)^\frac{q+6cq-5-5c}{5(1+c)(q-2)}\nonumber\\
 			&\cdot  \Big(\int_{B_{k+1}^+}|u|^q \ dx\Big)^\frac{3-7c}{5(1+c)(q-2)},
 		\end{align}
 		where  the exponents $\frac{q+6qc-5-5c}{5(1+c)(q-2)}$ and $\frac{3-7c}{5(1+c)(q-2)}$ are positive due to the fact  $ c\in \big(\frac{5-q}{6q-5}, \frac14\big)$.
 		
 		To deal with the first term on the right-hand side of  \eqref{4.9}, we decompose  the velocity $u$  and the pressure $p$ as follows
 		\begin{align*}
 			u=v_k+w_k,\quad p=p_k+h_k,
 		\end{align*}
 		where  $(v_k, p_k)$ satisfy the boundary value problem
 		\begin{align*}
 			\begin{cases}
 				-\Delta v_k+\nabla p_k=-u\cdot\nabla u +f \quad \mathrm{in}\ \tilde B_k,\\
 				\nabla \cdot v_k=0 \quad \mathrm{in}\  \tilde B_k,\\
 				(p_k)_{\tilde B_k}=0 ,\\
 				v_k=0 \quad \mathrm{on}\ \partial  \tilde B_k.\\
 			\end{cases}
 		\end{align*}
 		By  the assumptions in Theorem \ref{Them 2}, \eqref{4.1} and the  stationary case of the estimate in Lemma 4.4  of \cite{HK},   one has
 		\begin{align} \label{4.10}
 			&\big\||v_k|+|p_k|+|\nabla p_k|\big\|_{L^{1+c}(\tilde B_k)}\nonumber\\
 			&\leq 2^{bk}N\Big[\|u\cdot\nabla u\|_{L^{1+c}(\tilde B_k)}+ \| f\|_{L^{1+c}(\tilde B_k)}\Big]\nonumber\\
 			&\leq  2^{bk} N\Big[\|u\|_{L^2(\tilde B_k)}^{\frac{1-4c}{1+c}}\|\nabla u\|_{L^2(\tilde B_k)}^{\frac{1
 					+6c}{1+c}} + \|u\|_{L^2(\tilde B_k)}\|\nabla u\|_{L^2(\tilde B_k)}+   \| f\|_{L^2(\tilde B_k)}\Big]\nonumber\\
 			&\leq 2^{bk}N\Big[
 			A^+(\rho_{k+2})^\frac{1-4c}{2(1+c)}E^+(\rho_{k+2})^\frac{1+6c}{2(1+c)}+ A^+(\rho_{k+2})^\frac12E^+(\rho_{k+2})^{\frac12} + F^+(\rho_{k+2})^\frac12\Big],
 		\end{align}
 		where  the positive constants  $b$ and $N$ depend only on $c$.
 		
 		Since  $(w_k, h_k)$ is a suitable weak solutions to the following equations
 		\begin{align*}
 			\begin{cases}
 				-\Delta w_k+\nabla h_k=0 \quad \mathrm{in}\   \tilde B_k,\\
 				\nabla\cdot w_k=0\quad \mathrm{in}\ \tilde  B_k,\\
 				w_k=0 \quad \mathrm{on}\ \partial  B_k\cap \{x^5= 0\},
 			\end{cases}
 		\end{align*}
 		by \eqref{4.10} and  the stationary case of  the estimate in  Lemma 4.5 of \cite{HK},  we have
 		\begin{align}\label{4.11}
 			&\| h_k-(h_k)_{B^+_{k+1}}\|_{L^{1+c}(B_{k+1}^+)}+\|\nabla h_k\|_{L^{1+c}(B_{k+1}^+)}\nonumber\\
 			& \leq  2^{bk}N\Big(\|w\|_{L^1(B_{k+2}^+)}+\|h_k-(h_k)_{B^+_{k+2}}\|_{L^1(B_{k+2}^+)}\Big)\nonumber\\
 			&\leq  2^{bk}N\Big(\|u\|_{L^1(B_{k+2}^+)}+\|v\|_{L^1(B_{k+2}^+)}+\|p-(p)_{ B^+_{k+2}}\|_{L^1(B_{k+2}^+)}+\|p_k\|_{L^1(B_{k+2}^+)}\Big)
 			\nonumber\\
 			&\leq  2^{bk}N\Big(  C^+(\rho_{k+3})^\frac{1}{q}+
 			A^+(\rho_{k+3})^\frac{1-4c}{2(1+c)}E^+(\rho_{k+3})^\frac{1+6c}{2(1+c)}
 			+ A^+(\rho_{k+3})^\frac12E^+(\rho_{k+3})^{\frac12}\nonumber\\
 			&\quad+ F^+(\rho_{k+3})^\frac12 + P^+(\rho_{k+3}) \Big),
 		\end{align}
 		where  $b$ and $N$ depend only on $c$.
 		
 		The combination of \eqref{4.9}, \eqref{4.10} and  \eqref{4.11} yields
 		\begin{align*}
 			& \int_{B_{k+1}^+} |p-(p)_{B^+_{k+1}}  ||u|\ dx\\
 			&\leq  2^{bk}N \Big(   C^+(\rho_{k+3})^\frac{1}{q}+   A^+(\rho_{k+3})^\frac{1-4c}{2(1+c)}E^+(\rho_{k+3})^\frac{1+6c}{2(1+c)}+ A^+(\rho_{k+3})^\frac12E^+(\rho_{k+3})^{\frac12}\\
 			&\quad+ F^+(\rho_{k+3})^\frac12  +  P^+(\rho_{k+3})\Big) \cdot A^+(\rho_{k+3})^\frac{q+6cq-5-5c}{5(1+c)(q-2)}C^+(\rho_{k+3})^\frac{3-7c}{5(1+c)(q-2)}\\
 			& \leq 2^{bk}N \Big(\varepsilon^{\frac{2}{q}}+\varepsilon^{\frac{2-3c}{q(1+c)}}E^+(\rho_{k+3})^\frac{1+6c}{2(1+c)}+\varepsilon^\frac{2}{q}E^+(\rho_{k+3})^\frac12+\varepsilon^{\frac{1}{2}+\frac{1}{q}}+\varepsilon^{1+\frac{1}{q}}\Big)\\
 			&\leq 2^{bk}N \Big(\varepsilon^{\frac{2-3c}{q(1+c)}}E^+(\rho_{k+3})^\frac{1+6c}{2(1+c)}+\varepsilon^\frac{2}{q}E^+(\rho_{k+3})^\frac12 +\varepsilon^{\frac{2}{q}}\Big),
 		\end{align*}
 		where   $b$ and $N$ depend only on $c$.
 		
 		Inserting   \eqref{4.6}-\eqref{4.8} and the  above inequality into \eqref{4.5}, one derives
 		\begin{align}\label{4.12}
 			E^+(\rho_k)\leq& 2^{(2+b)k}N\Big(\varepsilon^{\frac{q+5aq-10a}{2q^2}}E^+(\rho_{k+3})^\frac{5(q-aq+2a)}{4q}+ \varepsilon^{ \frac{3}{q}}+\varepsilon^{\frac{2-3c}{q(1+c)}}E^+(\rho_{k+3})^{\frac{1+6c}{2(1+c)}}\nonumber\\
 			&+\varepsilon^\frac{2}{q}E^+(\rho_{k+3})^\frac12+\varepsilon^{\frac{2}{q}}+ \varepsilon^{\frac{q+5aq-10a}{3q^2}}E^+(\rho_{k+3})^\frac{5(q-aq+2a)}{6q}+\varepsilon^\frac{2}{q}+\varepsilon^{\frac12+\frac{1}{q}}\Big)\nonumber\\
 			\leq& 2^{(2+b)k}N\Big(\varepsilon^{\frac{q+5aq-10a}{2q^2}}E^+(\rho_{k+3})^\frac{5(q-aq+2a)}{4q}+\varepsilon^{\frac{2-3c}{q(1+c)}}E^+(\rho_{k+3})^{\frac{1+6c}{2(1+c)}}\nonumber\\
 			&+\varepsilon^\frac{2}{q}E^+(\rho_{k+3})^\frac12+ \varepsilon^{\frac{q+5aq-10a}{3q^2}}E^+(\rho_{k+3})^\frac{5(q-aq+2a)}{6q}+\varepsilon^{\frac{2}{q}} \Big),
 		\end{align}
 		where $b=b(c)$ and $N=N(a, c, q)$ are  some positive constants.
 		Since   $ c\in \Big(\frac{5-q}{6q-5}, \frac14\Big)$, if we take $ a\in \Big(\frac{q}{5q-10}, 1\Big)$,  then all the exponents of $E^+(\rho_{k+3})$ are   positive and smaller than one.
 		Hence, by \eqref{4.12} and Young's inequality, for any $\delta>0$, we have
 		\begin{align*}
 			E^+(\rho_k)\leq \delta^3 E^+(\rho_{k+3})+2^{(2+b)k} N\varepsilon^\beta,
 		\end{align*}
 		where $  N= N(c, q, \delta)$  and $\beta=\beta(c, q)$ are some positive constants.
 		
 		Let $\delta=3^{-(2+b)}$.   Multiplying both sides of  the  above inequality by $\delta^k$ and summing over $k$ from $1$ to infinity, we have
 		\begin{align*}
 			\sum_{k=1}^\infty \delta^k E^+(\rho_k)\leq \sum_{k=4}^\infty  \delta^{k}E^+(\rho_k)+   N\varepsilon^\beta \sum_{k=1}^\infty\Big(\frac{2}{3}\Big)^{(2+b)k},
 		\end{align*}
 		which implies
 		\begin{align*}
 			E^+(15/16)\leq  N\varepsilon^\beta.
 		\end{align*}
 		The lemma is proved.
 	\end{proof}
 	
 	In the following  lemma, we will   verify the decay estimates of $A^+(x, \rho)+E^+(x, \rho)$, $G^+(x, \rho)$, and  $P^+(x, \rho)$ with respect to  radius  based on the above three lemmas and the condition \eqref{1.4}. This is a key lemma for us to prove Theorem \ref{Them 2}.
 	\begin{lemma}\label{lem 4.4}
 		There exists a constant $ \varepsilon>0$ satisfying the following property.  If
 		\begin{align}\label{4.14}
 			A^+(\rho_0)^\frac{2-l}{2(1+l)}+  E^+(\rho_0)^\frac{2-l}{2(1+l)}+F^+(\rho_0)^\frac{2-l}{2(1+l)}+P^+(\rho_0) \leq  \varepsilon,
 		\end{align}
 		then there exists sufficiently small $l\in\Big(0, \frac{4q-10}{16q-25}\Big)$ such that for any  $x\in \overline{ B_{ \rho_0/5}^+} $ and $\rho\in \big(0,
 		\rho_0/5\big)$,  we have
 		\begin{align}\label{4.15}
 			A^+(x, \rho)+  E^+(x, \rho)  \leq N \varepsilon^{\frac{(2-l)(2q-5-ql)}{5q-10}} \Big( \frac{\rho}{\rho_0}\Big) ^{\frac{(2-l)^2}{2}},
 		\end{align}
 		\begin{align}\label{4.16}
 			G^+(x, \rho)  \leq  N  \varepsilon^\frac{(1+c)(1+l)(2q-5-ql)}{(q-2)(4-c)}
 			\Big(\frac{\rho}{\rho_0}\Big)^\frac{5(1+c)(1+l)(2-l)}{2(4-c)},
 		\end{align}
 		and
 		\begin{align}\label{4.17}
 			P^+(x,\rho) \leq N \varepsilon^{\frac{(1+l)(2q-5-ql)}{5q-10}} \Big(\frac{\rho}{\rho_0}\Big) ^{\frac{(1+l)(2-l)}{2}},
 		\end{align}
 		where $N$ is a positive constant depending on $c$, $l$, and $r'\in (1+c, +\infty)$ in Lemma \ref{lem 4.2}, but independent of $ \varepsilon$, $\rho_0$, and $ \rho$.
 	\end{lemma}
 	\begin{proof}
 		We prove this lemma by  an iteration argument.
 		We discuss two cases according to the position of $x$.
 		We first consider the case when  $x \in  B_{\rho_0/5}\cap \{x^5=0\}$  and  derive  the  decay estimates of  scale invariant quantities. Then, by another iteration argument,  we extend  previous results to  the general  case   $x\in B^+_{\rho_0/5}$.
 		
 		Similar to the  interior case, we denote $\rho_k=\tilde\rho^{(1+\beta)^k}$, where   $\tilde\rho=3\rho_0/5$ and $\beta>0$ is some  small constant which will be specified later.
 		
 		{\bf Case 1:  $x\in B_{\rho_0/5}\cap \{x_5=0\}$.}
 		We fix the  parameter $\alpha\in (2-l, 2)$.  By the   scale invariant property, we first assume $\tilde\rho^{\alpha }= N \varepsilon$.  We only need to prove the following decay estimates
 		\begin{align}\label{4.18}
 			A^+(x, \rho_k)+ E^+(x, \rho_k)  \leq  \rho_k^{\alpha}, \   G^+(x, \rho_k)\leq \rho_k^{\frac{5\alpha(1+c)(1+l)}{(4-c)(2-l)}},\  P^+(x, \rho_k)\leq \rho_k^{\frac{\alpha(1+l)}{2-l}}   .
 		\end{align}
 		
 		Let us prove the above estimates for  $k=0$.
 		Since $B^+(x, \tilde\rho)\subset B^+_{\rho_0},$ by  \eqref{4.14}, we have
 		\begin{align}\label{4.19}
 			A^+(x, \tilde\rho) \leq NA^+(\rho_0) \leq   N \varepsilon^\frac{2(1+l)}{2-l}\leq N \varepsilon=\tilde\rho^\alpha
 		\end{align}
 		and
 		\begin{align}\label{4.20}
 			E^+(x, \tilde\rho) \leq NE^+(\rho_0) \leq   N \varepsilon^\frac{2(1+l)}{2-l}\leq N \varepsilon=\tilde\rho^\alpha.
 		\end{align}
 		For  $G^+(x, \tilde \rho)$ and $P^+(x, \tilde\rho)$,   since  $B^+(x, 4\tilde\rho/3)\subset B^+_{\rho_0}$, by Lemma \ref{lem 4.2},  \eqref{4.14} and H\"older's inequality, we  obtain
 		\begin{align}\label{4.21}
 			G^+(x, \tilde\rho)\leq& N \Big( E^+(x, 4\tilde\rho/3)^{\frac{5(1+c)}{4-c}}
 			+F^+(x, 4\tilde\rho/3)^{\frac{5(1+c)}{2(4-c)}}\notag \\
 			&\quad +E^+(x, 4\tilde\rho/3)^{\frac{5(1+c)}{2(4-c)}} +P^+(x, 4\tilde\rho/3)^{\frac{5(1+c)}{4-c}}  \Big)\nonumber\\
 			\leq&  N \Big( E^+(\rho_0)^{\frac{5(1+c)}{4-c}}
 			+F^+(\rho_0)^{\frac{5(1+c)}{2(4-c)}}+  E^+(\rho_0)^{\frac{5(1+c)}{2(4-c)}} +P^+(\rho_0)^{\frac{5(1+c)}{4-c}}  \Big)\nonumber\\
 			\leq& N\Big( \varepsilon^{\frac{2(1+l)}{2-l}\cdot\frac{5(1+c)}{4-c} } + \varepsilon^{\frac{2(1+l)}{2-l}\cdot\frac{5(1+c)}{2(4-c)}} + \varepsilon^{\frac{2(1+l)}{2-l}\cdot\frac{5(1+c)}{2(4-c)}}+ \varepsilon^\frac{5(1+c)}{4-c}\Big)\nonumber\\
 			\leq &(N \varepsilon)^{\frac{5(1+c)(1+l)}{(2-l)(4-c)}} = \tilde\rho^{\frac{5\alpha(1+c)(1+l)}{(2-l)(4-c)}}
 		\end{align}
 		and
 		\begin{align}\label{4.22}
 			P^+(x,\tilde\rho)\leq N G^+(x, \tilde\rho)^\frac{4-c}{5(1+c)}\leq (N \varepsilon)^\frac{(1+l) }{2-l}= \tilde\rho^\frac{\alpha(1+l)}{2-l},
 		\end{align}
 		where $N=N(c, l, r')$.
 		Combining \eqref{4.19}-\eqref{4.22},  we get \eqref{4.18} for $k=0$.
 		
 		To prove \eqref{4.18} for $k>0$, since $A^+(x, \rho_k)+E^+(x, \rho_k)$  and $P^+(x, \rho_k)$ can be estimated by  the same method as in Lemma \ref{lem 3.4},  we only consider $G^+(x, \rho_k)$. We suppose that \eqref{4.18} holds for $k=1,\ldots,k_0\geq m$, where $m$ is an integer to be specified later.  We will  show that  the estimate of $G^+$ also holds for $\rho_{k_0+1}$.
 		
 		Taking $\tilde \beta=(1+\beta)^{m+1}-1$, $\gamma=\rho_{k_0-m}^{\tilde\beta}$, and $\rho=\rho_{k_0-m}$, by  Lemma \ref{lem 4.2}, we have
 		\begin{align}\label{4.23}
 			G^+(x, \rho_{k_0+1})\leq& N\Big[\rho_{k_0-m}^{-\frac{5(2-3c)\tilde\beta}{4-c}+\frac{5\alpha(1-4c) }{2(4-c)}+\frac{5\alpha(1+6c)}{2(4-c)}}+\rho_{k_0-m}^{-\frac{5(2-3c)\tilde\beta}{4-c} +\frac{5\alpha(1+c)}{4-c} }\nonumber\\
 			&+\rho_{k_0-m}^{-\frac{5(2-3c)\tilde\beta}{4-c}+\frac{5(1+c)(3q-5)}{q(4-c)}}\|f\|_{L^q}^\frac{5(1+c)}{4-c}\nonumber\\
 			&+\rho_{k_0-m}^{-\frac{5(2-3c)\tilde\beta}{4-c}+25\tilde\beta [\frac{1}{4-c}-\frac{1+c}{r'(4-c)}]}\Big(\rho_{k_0-m}^{\frac{5\alpha(1+c)}{2(4-c)}}
 			+\rho_{k_0-m}^{\frac{5\alpha(1+c)(1+l)}{(4-c)(2-l)}}\Big)\Big],
 		\end{align}
 		where $N=N(c, l, r')$. If we choose
 		\begin{align}\label{4.24}
 			\frac{3\alpha l}{(6-10/r')(2-l)-2\alpha(1+l)}<\tilde\beta< \frac{(1+c)[6q-\alpha q-10-(3q+\alpha q-5)l]}{q[(1+c)(1+l)\alpha+(2-3c)(2-l)]},
 		\end{align}
 		then all  the exponents on the right-hand side of  \eqref{4.23} are greater than $$\frac{5\alpha(1+c)(1+l)(1+\tilde\beta)}{(4-c)(2-l)}.$$
 		Here, it is sufficient to take a small   $l\in\Big(0, \frac{ 4q-10 }{16q-25}\Big). $ Indeed we can choose $\beta=l^2$ and take a sufficiently large integer $m$ of order $1/l$ so that $\tilde\beta\sim l$. A simple calculation shows that \eqref{4.24} hold.
 		
 		Hence,  there exists  a  small constant $\xi>0$ such that
 		\begin{align}\label{4.26}
 			G^+(x, \rho_{k_0+1})\leq N\rho_{k_0-m}^{\big[\frac{5\alpha(1+c)(1+l)}{(4-c)(2-l)}+\xi\big](1+\tilde\beta)}\leq N \rho_{k_0+1}^{\frac{5\alpha(1+c)(1+l)}{(4-c)(2-l)}+\xi}.
 		\end{align}
 		By taking   $ \varepsilon>0$ sufficiently small,  one has
 		\begin{align*}
 			N\rho_{k_0+1}^\xi< N \tilde\rho^\xi\leq N(N \varepsilon)^\frac{\xi}{2}<1.
 		\end{align*}
 		Inserting the  above inequality into \eqref{4.26}, we obtain
 		\begin{align*}
 			G^+(x, \rho_{k_0+1})\leq \rho_{k_0+1}^{\frac{5\alpha(1+c)(1+l)}{(4-c)(2-l)}}.
 		\end{align*}
 		Thus, the estimate of $G^+$ in \eqref{4.18} holds for $k=k_0+1$ provided that it holds for $k=1,\ldots,k_0$.
 		Now analogous to the case $k=0$, by choosing $ \varepsilon$ sufficiently small, we can obtain   \eqref{4.18} for   $k=1,\dots, m$. Hence, by induction we have \eqref{4.18} for any integer $k\ge 0$.
 		
 		For any $\rho\in \big(0, 3\rho_0/5\big)$,  there exists a positive integer $k$ such that $\rho_{k+1}\leq\rho<\rho_k$. Similar to the calculations in   \eqref{3.40}-\eqref{3.43}, we  have
 		\begin{align*}
 			A^+(x, \rho)\leq \rho^{2-l},\ \  E^+(x, \rho)\leq   \rho^{2-l},  \ \ G^+(x, \rho)\leq \rho^\frac{5(1+c)(1+l)}{4-c}, \ \ P^+(x,  \rho)\leq \rho^{1+l}.
 		\end{align*}
 		Hence, \eqref{4.18} are proved  for any  $x\in B_{\rho_0/5}\cap\{x_5=0\}$  when  $\tilde\rho^\alpha=N  \varepsilon$.
 		
 		For the general case, due to the scale invariant property, we derive the decay estimates of all the scale invariant quantities by imposing  additional scaling factors on the right-hand side of \eqref{4.18}. To be precise,    if $(u, p)$ is a suitable weak solution to \eqref{1.1} in $  B^+(x, \rho_0)$, by taking $\lambda=\rho_0/(N \varepsilon)^{\frac{1}{\alpha}}$,
 		we know that $(u_\lambda, p_\lambda)$ is also a suitable weak solution to \eqref{1.1} in $  B^+\big(x, (N \varepsilon)^\frac{1}{\alpha}\big) $. Hence,  similar to the interior case, we have
 		\begin{align}\label{4.28}
 			A^+(x, \rho)+E^+(x, \rho)\leq N \varepsilon^\frac{2-l}{\alpha}\Big(  \frac{\rho}{\rho_0}\Big)^{2-l} \leq N \varepsilon^\frac{2-l}{2}\Big(  \frac{\rho}{\rho_0}\Big)^{2-l},
 		\end{align}
 		\begin{align}\label{4.29}
 			G^+(x, \rho)\leq N \varepsilon^\frac{5(1+c)(1+l)}{\alpha(4-c)} \Big(\frac{\rho}{\rho_0}\Big)^\frac{5(1+c)(1+l)}{4-c} \leq N \varepsilon^\frac{5(1+c)(1+l)}{2(4-c)} \Big(\frac{\rho}{\rho_0}\Big)^\frac{5(1+c)(1+l)}{4-c},
 		\end{align}
 		and
 		\begin{align}\label{4.30}
 			P^+(x,  \rho)\leq N \varepsilon^\frac{1+l}{\alpha}\Big(\frac{\rho}{\rho_0}\Big)^{1+l} \leq N \varepsilon^\frac{1+l}{2}\Big(\frac{\rho}{\rho_0}\Big)^{1+l}, \end{align}
 		where  $N$ is a positive constant depending on $c, l$ and $ r'$.

 		{\bf Case 2:  $x\in \overline{B^+_{ \rho_0/5}}$.}
 		By comparing  $d_x$, the distance of $x$ to the boundary $\{x_5=0\}$, with  $\rho\in\big(0, \rho_0/5\big)$, the radius of  the ball around $x$, we further consider two cases to derive the decay estimates of scale invariant quantities. We denote the projection of $x$ on the boundary by $ x^*$.
 		
 		{\bf Case 2.1:  $ \rho\geq d_x/2$.}    Since $B^+ (x, \rho)\subset B^+\big( x^*,  3 \rho\big)$, we have
 		\begin{align*}
 			A^+(x, \rho)\leq NA^+\big( x^*,  3\rho\big), \quad E^+(x, \rho)\leq NE^+\big( x^*,  3\rho\big),\\
 			G^+(x, \rho)\leq NG^+(  x^*, 3\rho),\quad P^+(x, \rho)\leq N P^+(  x^*,  3\rho).
 		\end{align*}
 		Thus, by the conclusions of the boundary decay estimates \eqref{4.28}-\eqref{4.30}, we have
 		\begin{align*}
 			A^+(x,\rho)+E^+(x, \rho)  \leq N  \varepsilon^{\frac{2-l}{2}}  \Big( \frac{\rho}{\rho_0}\Big) ^{2-l } ,
 		\end{align*}
 		\begin{align*}
 			G^+(x,\rho) \leq N  \varepsilon^\frac{5(1+c)(1+l)}{2(4-c)}  \Big( \frac{\rho}{\rho_0}\Big) ^{\frac{5(1+c)(1+l)}{4-c}},
 		\end{align*}
 		and
 		\begin{align*}
 			P^+(x,\rho) \leq  N \varepsilon^{\frac{1+l}{2}}  \Big( \frac{\rho}{\rho_0}\Big) ^{1+l},
 		\end{align*}
 		where $N=N(c, l, r')$.
 		
 		{\bf  Case 2.2:  $\rho< d_x/2$.}
 		Since $$B^+(x, \rho)=B(x, \rho)\subset B(x, d_x)\subset B^+( x^*,  2 d_x), $$      applying the boundary results \eqref{4.28} and \eqref{4.30}  to $x^*$, we  have
 		\begin{align}\label{4.31}
 			A(x, d_x)+  E(x, d_x)\leq N \Big(A^+(x^*, 2d_x)+  E^+(x^*, 2d_x)\Big)\leq N \varepsilon^{\frac{2-l}{2}}  \Big(\frac{2d_x}{\rho_0} \Big)^{2-l}
 		\end{align}
 		and
 		\begin{align}\label{4.32}
 			P(x, d_x)\leq N P^+(x^*, 2d_x)\leq N \varepsilon^{\frac{1+l}{2}} \Big(\frac{2d_x}{\rho_0}\Big)^{1+l},
 		\end{align}
 		where $N=N(c, l, r)$.
 		Combining  \eqref{4.31} and \eqref{4.32},  we obtain
 		\begin{align}\label{AEP}
 			&A(x, d_x)+  E(x, d_x)+ P(x, d_x)^\frac{2-l}{1+l}\notag\\
 			& \leq N \varepsilon^\frac{2-l}{2}\Big(\frac{2d_x}{\rho_0}\Big)^{2-l}\leq N \varepsilon^\frac{4q-10-2ql}{5q-10}\Big(\frac{2d_x}{\rho_0}\Big)^{2-l}.
 		\end{align}
 		
 		For    $F(x, d_x)$,  when $\Big(\frac{2d_x}{\rho_0}\Big)< \varepsilon^\frac{2q(1+l)}{5(q-2)(2-l)}$, by  H\"older's inequality, one has
 		\begin{align}\label{F1}
 			F(x, d_x)=F^+(x, d_x)\leq N d_x^{6-\frac{10}{q}}\|f\|_{L^q(B_{1}^+)}^2.
 		\end{align}
 		Otherwise, by the condition \eqref{4.14},  one has
 		\begin{align}\label{4.34}
 			F(x, d_x)\leq NF^+(x^*, 2 d_x)\leq N \Big(\frac{2d_x}{\rho_0}\Big)F^+(\rho_0)\leq N  \varepsilon^\frac{2(1+l)}{2-l} \Big(\frac{2d_x}{\rho_0}\Big).
 		\end{align}
 		Since $\rho_0\leq 1$,  by \eqref{F1} and \eqref{4.34}, we obtain
 		\begin{align}\label{F2}
 			F(x, d_x)^\frac{2-l}{2(1+l)}\leq N \varepsilon^\frac{4q-10-2ql}{5q-10}\Big(\frac{2d_x}{\rho_0}\Big)^{2-l}.
 		\end{align}
 		Thus,  by \eqref{AEP} and \eqref{F2},  if we choose $ \varepsilon$   sufficiently small such that
 		\begin{align*}
 			N \varepsilon^{\frac{4q-10-2ql}{5q-10}}\Big(\frac{2d_x}{\rho_0}\Big)^{2-l}\leq \tilde\varepsilon,
 		\end{align*}
 		where $\tilde\varepsilon$ is from  Lemma \ref{lem 3.4}, then   the   condition \eqref{3.20} in Lemma \ref{lem 3.4} is satisfied. Hence, from \eqref{3.21}-\eqref{3.23},  we have
 		\begin{align*}
 			A^+(x, \rho)+  E^+(x, \rho) =&A(x, \rho)+  E(x, \rho)\notag\\
 			\leq& N \Big[ \varepsilon^{\frac{4q-10-2ql}{5q-10}}\Big(\frac{2d_x}{\rho_0}\Big)^{2-l}\Big]^\frac{2-l}{2}\Big(\frac{\rho}{2d_x}\Big)^{2-l}\notag\\
 			\leq& N \varepsilon^{\frac{(2-l)(2q-5-ql)}{5q-10}} \Big( \frac{\rho}{\rho_0}\Big) ^{\frac{(2-l)^2}{2}},
 		\end{align*}
 		\begin{align*}
 			G^+(x, \rho)=G(x, \rho)&\leq N \Big[ \varepsilon^{\frac{4q-10-2ql}{5q-10}}\Big(\frac{2d_x}{\rho_0}\Big)^{2-l}\Big]^\frac{5(1+c)(1+l)}{2(4-c)}\Big(\frac{\rho}{2d_x}\Big)^\frac{5(1+c)(1+l)}{4-c}\notag\\
 			&\leq N  \varepsilon^\frac{(1+c)(1+l)(2q-5-ql)}{(q-2)(4-c)}\Big(\frac{\rho}{\rho_0}\Big)^\frac{5(1+c)(1+l)(2-l)}{2(4-c)},
 		\end{align*}
 		as well as
 		\begin{align*}
 			P^+(x,\rho)= P(x,\rho)\leq& N \Big[ \varepsilon^{\frac{4q-10-2ql}{5q-10}}\Big(\frac{2d_x}{\rho_0}\Big)^{2-l}\Big]^\frac{1+l}{2}\Big(\frac{\rho}{2d_x}\Big)^{1+l}\notag\\
 			\leq &N \varepsilon^{\frac{(1+l)(2q-5-ql)}{5q-10}} \Big(\frac{\rho}{\rho_0}\Big) ^{\frac{(1+l)(2-l)}{2}},
 		\end{align*}
 		where $N=N(c, l, r')$.
 		The lemma is proved.
 	\end{proof}
 	
 	The rest of this  section is devoted to the proof of Theorem \ref{Them 2}. We will use Lemma \ref{lem 4.4} to prove  that $u$ is H\"older continuous in   $\overline{B^+_{1/2}}$ by  Campanato's characterization of H\"older continuity and   a covering argument. By  the   condition \eqref{1.4},  H\"older's inequality which gives
 	\begin{align*}
 		A^+(1) \leq NC^+(1)^\frac{2}{q}\leq N\varepsilon^{\frac{2}{q}}
 	\end{align*}
 	and Lemma \ref{lem 4.3},
 	we can choose $\varepsilon>0$ sufficiently small  such that the condition \eqref{4.14}  holds with $\rho_0=\frac{15}{16}$. Hence,  by Lemma \ref{lem 4.4}, we obtain  the  decay estimates  \eqref{4.15}-\eqref{4.17}. Let  $x\in\overline{ B_{3/16}^+}$ and  $\rho_0=\frac{15}{16}$. According to the position of $x$, we discuss two cases to prove the regularity of $u$.
 	
 	Let us first consider the case when   $x\in B_{3/16}\cap \{x^5=0\}$.
 	We decompose the suitable weak solution $u$ of \eqref{1.1} as  $ u= w +v$,  where $w$ satisfies the equation
 	\begin{align}\label{weq}
 		\Delta w_i=\partial_i\big(p-(p)_{B^+(x,  \rho)}\big)+\partial_j\big(u_iu_j\big)+f_i\quad \text{in}\ B^+(x,   \rho)
 	\end{align}
 	with the zero Dirichlet boundary condition,    where  $\rho\in(0, 3/16)$.
 	By the $L^p$ estimate for elliptic equations, we have
 	\begin{align}\label{w}
 		&\|\nabla w\|_{L^\frac{10}{7}(B^+(x,   \rho))}\leq N\Big(\big\| p-(p)_{B^+(x,  \rho)}\big\|_{L^\frac{10}{7}(B^+(x, \rho))}\notag\\
 		&\qquad +\big\||u|^2\big\|_{L^\frac{10}{7}(B^+(x,  \rho))}+ \rho\|f\|_{L^\frac{10}{7}(B^+(x, \rho))}\Big).
 	\end{align}
 	
 	
 	We first estimate the pressure term on the right-hand side of  \eqref{w}.  Due to \eqref{4.16} and \eqref{4.17},  we have
 	\begin{align*}
 		G^+(x, \rho)\leq N \rho^\frac{5(1+c)(1+l)(2-l)}{2(4-c)},\quad
 		P^+(x, \rho)
 		\leq &N   \rho^{\frac{(1+l)(2-l)}{2}},
 	\end{align*}
 	where $N=N(c, l, r')>0$ is some constant.
 	Hence,   by H\"older's inequality and the above decay estimates,  we have
 	\begin{align}\label{4.37}
 		&\Big(\int_{B^+(x,\rho)}|p-(p)_{B^+(x, \rho)}|^\frac{10}{7}\ dx\Big)^\frac{7}{5}\notag\\
 		&\leq   \Big(\int_{B^+(x,\rho)}|p-(p)_{B^+(x, \rho)}| \ dx\Big)^\frac{9c-1}{1+6c}\Big(\int_{B^+(x,\rho)}|p-(p)_{B^+(x, \rho)}|^\frac{5(1+c)}{4-c}\ dx\Big)^\frac{3(4-c)}{5(1+6c)}\nonumber\\
 		&\leq  \rho^3P^+(x,\rho)^\frac{9c-1}{1+6c}G^+(x,\rho)^\frac{3(4-c)}{5(1+6c)}\leq  N   \rho^{5+l-l^2},
 	\end{align}
 	where $N=N(c, l, r')>0$ and  $\rho\in(0, 3/16)$ .
 	
 	For the second term on the right-hand side of \eqref{w},  due to  \eqref{4.15}, one has
 	\begin{align*}
 		A^+(x, \rho)+  E^+(x, \rho)
 		\leq N  \rho^{\frac{(2-l)^2}{2}},
 	\end{align*}
 	where $N=N(c, l, r')>0$ is some constant.
 	Thus,  by  the Sobolev embedding inequality and  the above decay rate, we derive
 	\begin{align}\label{4.36}
 		&\Big( \int_{B^+(x,\rho)}|u|^\frac{20}{7} \ dx\Big)^\frac{7}{5}\nonumber\\
 		&\leq N  \Big(\int_{B^+(x,\rho)}|u|^2 \ dx\Big)^\frac12\Big(\int_{B^+(x,\rho)}|\nabla u|^2 \ dx+\rho^{-2}\int_{B^+(x,\rho)}|u|^2 \ dx\Big)^\frac32\nonumber\\
 		&\leq N \rho^3  A^+(x,\rho)^\frac12\Big(E^+(x,\rho)+A^+(x,\rho)\Big)^\frac32\leq N   \rho^{7+l^2-4l},
 	\end{align}
 	where $N=N(c, l, r')>0$ and  $\rho\in(0, 3/16)$ .
 	
 	For  the last  term on the right-hand side of \eqref{w},  by H\"older's
 	inequality, one  has
 	\begin{align}\label{4.38}
 		\rho^2 \Big(\int_{B^+(x,\rho)}|f|^\frac{10}{7} \ dx\Big)^\frac{7}{5}\leq N \Big(\int_{B^+(x,\rho)}| f|^q \ dx\Big)^\frac{2}{q}\rho^{9-\frac{10}{q}},
 	\end{align}
 	where  $\rho\in(0, 3/16)$ is any constant.
 	
 	Thus, by the Sobolev-Poincar\'e  inequality and  \eqref{4.37}-\eqref{4.38}, we have
 	\begin{align}\label{4.39}
 		&\int_{B^+(x,\rho)}|w-(w)_{B^+(x, \rho)}|^2 \ dx\notag\\
 		&\leq N \Big(\int_{B^+(x, \rho)}|\nabla w|^\frac{10}{7} \ dx\Big)^\frac{7}{5}\nonumber\\
 		&\leq N\Big[ \Big(\int_{B^+(x,\rho)}|p-(p)_{B^+(x, \rho)}|^\frac{10}{7}\ dx\Big)^\frac{7}{5}  +   \Big( \int_{B^+(x,\rho)}|u|^\frac{20}{7} \ dx\Big)^\frac{7}{5} \nonumber\\
 		&\quad + \rho^2 \Big(\int_{B^+(x,\rho)}|f|^\frac{10}{7} \ dx\Big)^\frac{7}{5}\Big]\nonumber\\
 		&\leq N\Big(   \rho^{ 5+l-l^2}+   \rho^{7+l^2-4l}+ \rho^{9-\frac{10}{q}}\|f\|_{L^q(B^+(x,  \rho))}^2\Big),
 	\end{align}
 	where $N=N(c, l, r')>0$.
 	
 	Due to the boundary Poincar\'e inequality   and the fact  that  the $L^\infty$  norm of the gradient of harmonic function $v$ in $B^+(x, \gamma\rho)$ can be controlled by its $L^p$ norm    in  $B^+(x, \rho)$ for any $p\in[1, +\infty]$,  we obtain
 	\begin{align}\label{4.40}
 		\int_{B^+(x,\gamma\rho)}|v-(v)_{B^+(x,\gamma\rho)}|^2 \ dx&\leq N (\gamma\rho)^2\int_{B^+(x, \gamma\rho)}|\nabla v|^2 \ dx\nonumber\\
 		&\leq N  (\gamma\rho)^7\|\nabla v\|_{L^\infty(B^+(x,\gamma\rho))} ^2\nonumber\\
 		&\leq  N \gamma^7  \int_{B^+(x, \rho)}|v-(v)_{B^+(x, \rho)}|^2 \ dx,
 	\end{align}
 	where $\gamma\in(0, 1/2)$.
 	
 	The combination of \eqref{4.39} and \eqref{4.40} implies
 	\begin{align}\label{4.41}
 		&\int_{B^+(x,\gamma\rho)}|u-(u)_{B^+(x,\gamma\rho)}|^2 \ dx\notag\\
 		&\leq   \int_{B^+(x,\gamma\rho)}|v-(v)_{B^+(x,\gamma\rho)}|^2 \ dx+ \int_{B^+(x,\gamma\rho)}|w-(w)_{B^+(x,\gamma\rho)}|^2 \ dx\nonumber\\
 		&\leq  N \Big(\gamma^7  \int_{B^+(x, \rho)}|v-(v)_{B^+(x, \rho)}|^2 \ dx+   \rho^{   5+l-l^2}+   \rho^{  7+l^2-4l}+ \rho^{9-\frac{10}{q}}\|f\|_{L^q(B^+(x, \rho))}^2\Big)\nonumber\\
 		&\leq  N \Big(\gamma^7  \int_{B^+(x, \rho)}|u-(u)_{B^+(x, \rho)}|^2 \ dx+  \rho^{   5+l-l^2}+  \rho^{ 7+l^2-4l}+\rho^{9-\frac{10}{q}}\|f\|_{L^q(B^+_1)}^2\Big),
 	\end{align}
 	where $N=N(c, l, r')>0$. 
 	By the condition of $l$ in Lemma \ref{lem 4.4} and by taking $l\in\Big(0, \frac{4q-10}{16q-25}\Big),$  we have
 	\begin{align*}
 		\min\big\{  5+l-l^2,7+l^2-4l, \ 9- 10/q\big\}>5+ l/2 .
 	\end{align*}
 	From \eqref{4.41} and Lemma \ref{lem 4.5}, we obtain
 	\begin{align}\label{4.43}
 		\int_{B^+(x,  \rho)}\big|u-(u)_{B^+(x,  \rho)}\big|^2 \ dx\leq  N  \rho^{5+l/2}
 	\end{align}
 	for any   $x\in B_{3/16}\cap \{x^5=0\}$ and $\rho\in (0, 3/16)$.
 	
 	Next we discuss the case when $x\in B^+_{3/16}$.
 	Let $x^*$ be the projection of $x$ on the boundary $\{x_5=0\}$ and  $d_x$ be the distance  between $x$  and the flat boundary. Let   $\rho\in(0, 1/16)$. 
 	According to the values of $d_x$ and $\rho$, we divide the proof into  two cases.
 	
 	{\bf  Case 1:   $\rho\geq d_x/2$.}  In this case, we have $B^+ (x, \rho)\subset B^+\big(x^*,  3\rho\big)$.  By  \eqref{4.43},  we derive
 	\begin{align}\label{4.44}
 		&\int_{B^+(x,\rho)}|u-(u)_{B^+(x,\rho)}|^2 \ dx\notag\\
 		&\leq N\int_{B^+(x^*,  3\rho)}|u-(u)_{B^+(x^*,  3\rho)}|^2 \ dx\leq N \rho^{5+l/2},
 	\end{align}
 	where $N=N(c, l, r')$.
 	
 	{\bf Case 2:  $\rho< d_x/2$.} In this case, by the decay estimates  \eqref{4.15}-\eqref{4.17} in Lemma \ref{lem 4.4},  we have
 	\begin{align*}
 		A(x, \rho)+  E(x, \rho)=A^+(x, \rho)+  E^+(x, \rho)
 		\leq N  \rho^{\frac{(2-l)^2}{2}},
 	\end{align*}
 	\begin{align*}
 		G(x, \rho)= G^+(x, \rho)\leq N \rho^\frac{5(1+c)(1+l)(2-l)}{2(4-c)},
 	\end{align*}
 	and
 	\begin{align*}
 		P(x,\rho) =P^+(x, \rho)
 		\leq &N   \rho^{\frac{(1+l)(2-l)}{2}},
 	\end{align*}
 	where $N=N(c, l, r')$.
 	
 	Similar to the calculations in \eqref{4.37}-\eqref{4.41}, by the above  decay estimates, we obtain
 	\begin{align*}
 		\int_{B^+(x,\rho)}|u-(u)_{B^+(x,\rho)}|^2 \ dx=&\int_{B(x, \rho)}|u-(u)_{B(x, \rho)}|^2 \ dx\\
 		\leq& N \Big(\frac{\rho}{d}\Big)^7  \int_{B(x, d)}|u-(u)_{B(x, d)}|^2 \ dx \\
 		&+N\Big(  d^{  5+l-l^2}+ d^{ 7+l^2-4l} +d^{9-\frac{10}{q}}\|f\|_{L^q(B_1)}^2\Big)
 	\end{align*}
 	for any $\rho< d <  d_x/2$,   where $N=N(c, l, r')$.
 	Thus,  by Lemma \ref{lem 4.5} and   the condition of $l$,  one has
 	\begin{align}\label{4.45}
 		&\int_{B^+(x,\rho)}|u-(u)_{B^+(x,\rho)}|^2 \ dx\notag\\
 		&\leq N \Big(\frac{\rho}{d_x}\Big)^{5+l/2}  \int_{B(x, d_x)}|u-(u)_{B(x, d_x)}|^2 \ dx +N\rho^{5+l/2},
 	\end{align}
 	where $N=N(c, l, r')$.
 	
 	For the first term on the right-hand side of  \eqref{4.45},  by   \eqref{4.44}, we reach
 	\begin{align*}
 		\int_{B(x, d_x)}|u-(u)_{B(x, d_x)}|^2 \ dx= \int_{B^+(x, d_x)}|u-(u)_{B^+(x, d_x)}|^2 \ dx\leq N d_x^{5+l/2}.
 	\end{align*}
 	Inserting the above inequality into \eqref{4.45}, for any  $x\in B_{3/16}^+$ and  $\rho\in (0, 1/16)$,  we have
 	\begin{align*}
 		\int_{B^+(x, \rho)}|u-(u)_{B^+(x, \rho)}|^2 \ dx\leq N\rho^{5+l/2},
 	\end{align*}
 	where $N=N(c, l, r')$.
 	Hence, by Campanato's characterization of H\"older continuity around the boundary,   we see that $u$ is H\"older continuous in $B^+_{1/4}$. The conclusion of Theorem \ref{Them 2}  follows by a covering argument.
 	
 	\section{Acknowledgement}
 	The author would like to thank Professor Hongjie Dong for many helpful discussions on this work.

 	
 \end{document}